\documentclass[10pt]{amsart}
\usepackage{amsmath, amsthm, amsfonts, amssymb}
\usepackage{mathrsfs}
\providecommand{\noopsort[1]{}}
\usepackage{bbm}
\usepackage{stmaryrd}
\usepackage{ifthen}
\numberwithin{equation}{section}
\usepackage{a4}
\usepackage{verbatim}
\usepackage{hyperref}

\newtheorem{thm}{Theorem}[section]
\newtheorem{cor}[thm]{Corollary}
\newtheorem{prop}[thm]{Proposition}
\newtheorem{lem}[thm]{Lemma}

\theoremstyle{remark}
\newtheorem{rem}[thm]{Remark}
\newtheorem{example}[thm]{Example}
\newtheorem{hyp}[thm]{Hypothesis}

\theoremstyle{definition}
\newtheorem{defn}[thm]{Definition}

\newcommand{\eps}{\varepsilon}
\newcommand{\weak}{\rightharpoonup}
\newcommand{\one}{\mathbbm{1}}
\newcommand{\OU}{\mathsf{ou}}

\newcommand{\bx}{\mathbf{x}}
\newcommand{\by}{\mathbf{y}}

\newcommand{\ip}[2]{\ifthenelse{\equal{#1}{}}{\mbox{$ [ \,\cdot\, , \, \cdot \, ] $}}{
\mbox{$ \left[ #1 \, , \, #2 \right]$}}}
\newcommand{\norm}[1]{\ifthenelse{\equal{#1}{}}{\mbox{$\|\cdot\|$}}{\mbox{$\| #1 \|$}}}

\newcommand{\dual}[2]{\ifthenelse{\equal{#1}{}}{\mbox{$ \langle \,\cdot\; , \; \cdot \, \rangle $}}{
\mbox{$ \langle #1   ,  #2 \rangle$}}}

\newcommand{\CR}{\mathbb{R}}
\newcommand{\CC}{\mathbb{C}}
\newcommand{\CN}{\mathbb{N}}

\newcommand{\PP}{\mathbf{P}}

\newcommand{\FF}{\mathbb{F}}
\renewcommand{\P}{\mathbb{P}}

\newcommand{\B}{\mathcal{B}}
\newcommand{\M}{\mathcal{M}}

\newcommand{\F}{\mathcal{F}}

\newcommand{\bS}{\mathbf{S}}
\newcommand{\bX}{\mathbf{X}}
\newcommand{\bY}{\mathbf{Y}}

\newcommand{\bM}{\mathbf{M}}
\newcommand{\cF}{\mathscr{F}}
\newcommand{\cL}{\mathscr{L}}
\newcommand{\cT}{\mathscr{T}}
\newcommand{\cR}{\mathscr{R}}

\newcommand{\cD}{\mathscr{D}}

\newcommand{\cP}{\mathcal{P}}
\newcommand{\scrP}{\mathscr{P}}
\newcommand{\cB}{\mathscr{B}}
\newcommand{\TV}{\mathrm{TV}}

\newcommand{\expect}{\mathbb{E}}

\newcommand{\half}{\frac{1}{2}}

\newcommand{\inject}{\hookrightarrow}

\renewcommand{\P}{{\mathbb P}}
\renewcommand{\Re}{\mathrm{Re}}

\newcommand{\OO}{\mathcal{O}}

\newcommand{\fX}{\mathfrak{X}}
\newcommand{\fV}{\mathfrak{V}}

\begin{document}
\title[Well-posedness of semilinear equations]{Perturbation of strong Feller semigroups and 
well-posedness of semilinear stochastic equations on Banach spaces}
\author{Markus C.\ Kunze}
\address{Delft Institute of Applied Mathematics, Delft University of Technology, P.O. Box 5031, 2600 GA Delft, The Netherlands}
\curraddr{Institute of Applied Analysis, University of Ulm, 89069 Ulm, Germany}
\email{markus.kunze@uni-ulm.de}
\thanks{The author was supported by VICI subsidy 639.033.604 in the `Vernieuwingsimpuls' program of the Netherlands Organization
for Scientific Research (NWO)}    

\begin{abstract}
We prove a Miyadera-Voigt type perturbation theorem for strong Feller semigroups. Using this result, we prove
well-posedness of the semilinear stochastic equation
\[ dX(t) = [AX(t) + F(X(t))]dt + GdW_H(t) \]
on a separable Banach space $E$, assuming that $F$ is bounded and measurable and that the associated linear 
equation, i.e.\ the equation with $F \equiv 0$,
is well-posed and its transition semigroup is strongly Feller and satisfies an appropriate gradient estimate. We also 
study existence and uniqueness of invariant measures for the associated transition semigroup. 
\end{abstract}

\maketitle 

\section{Introduction}
In this article we study the semilinear stochastic equation
\begin{equation}\label{eq.semilinear}
 dX(t)  = \big[AX(t) + F(X(t))\big]dt + GdW_H(t)
\end{equation}
on a real, separable Banach space $E$. Here $A$ generates a strongly continuous semigroup $\bS$ on $E$, $F$ is a bounded 
measurable map from $E$ to $E$ and $G \in \cL (H, E)$. In fact, we shall consider a more general situation and 
allow that $G$ maps $H$ into a larger Banach space $\tilde{E}$. The driving process $W_H$ is an $H$-cylindrical Wiener process.
In order to stress the dependence on the coefficients, we will refer to equation \eqref{eq.semilinear} 
as equation $[A, F, G]$.

We note that since $F$ is merely assumed to be measurable, the notions of existence and uniqueness of solutions 
to \eqref{eq.semilinear} have to be understood in a weak sense. If for every initial datum $x \in E$ there exists a unique
solution to \eqref{eq.semilinear}, then we say that the equation is well-posed. We will make this notion precise 
in Section \ref{sect.well}.\medskip 

In the case where $E$ is itself a Hilbert space, equation \eqref{eq.semilinear} has been studied by several 
authors. In the case where $\bS$ is compact and $G$ is of trace class, existence of solutions to \eqref{eq.semilinear} 
has been proved by G\c{a}tarek and Go{\l}dys \cite{gg94a} for continuous $F$. Afterwards, \cite{gg94}, they also proved 
uniqueness of solutions, assuming the strong Feller property and a suitable gradient estimate for the transition 
group $\cT_\OU$ of equation $[A,0,G]$.

Chojnowska-Michalik and Go{\l}dys \cite{cmg95} extended these results, dropping the assumption that $G$ is of trace class
and allowing weakly continuous functions $F$.\smallskip 

In this article, we generalize these results to general Banach spaces $E$ and arbitrary bounded, measurable functions 
$F$. In our main result below, $H_{Q_t}$ refers to the Cameron-Martin space of a Gaussian measure $\mu_t$ 
which appears in the transition semigroup $\cT_\OU$ of equation $[A, 0, G]$; we recall the definition and 
further properties of $H_{Q_t}$ in Section \ref{sect.perturb}. 

\begin{thm}\label{t.main}
Let $\tilde{E}$ be a real, separable Banach space and $H$ be a separable Hilbert space, $A$ be the generator of a 
strongly continuous semigroup $\bS$ on $\tilde{E}$ and $G \in \cL (H,\tilde{E})$. Moreover, let $E$ be a real, separable 
Banach space with $D(A) \subset E \subset \tilde{E}$ with continuous and dense embeddings.

We assume that equation $[A, 0, G]$ is well-posed on $E$
and that for every $t>0$ we have $S(t)E \subset H_{Q_t}$ with 
\begin{equation}\label{eq.hyp1} 
\int_0^T\|S(t)\|_{\cL (E, H_{Q_t})}\, dt < \infty 
\end{equation}
for all $T>0$. Finally, we assume that $S(t) \in \cL (\tilde{E}, E)$ for all $t>0$, that for $x \in \tilde{E}$ the 
$E$-valued map $t \mapsto S(t)x$ is continuous on $(0,\infty)$ and that for all $T>0$ we have
\begin{equation}\label{eq.hyp2}
 \int_0^T \|S(t)\|_{\cL (\tilde{E}, E)}^2\, dt < \infty\, .
\end{equation}

Then, for every bounded, measurable $F: E \to E$, equation $[A,F,G]$ is well-posed. Furthermore, the associated 
transition semigroup $\scrP$ is strongly Feller and irreducible.
\end{thm}

The estimate \eqref{eq.hyp1} implies that $\cT_\OU$ is strongly Feller and satisfies a gradient estimate. 
This is the natural generalization of the assumptions made in \cite{gg94, cmg95} to our more general setting.
On the other hand, estimate \eqref{eq.hyp2} is used in constructing solutions of equation $[A,F,G]$ from solutions 
of the associated local martingale problem.

These assumptions are satisfied in many important situations, in particular for the one-dimensional heat 
equation with space-time white noise. Examples are presented in Section \ref{sect.examples}.

We note that the assertion of existence of solutions for equation \eqref{eq.semilinear} for bounded, measurable 
$F$ appears to be new, even in the Hilbert space case. This is due to the fact that $G$ is not assumed to be invertible 
and hence we cannot invoke Girsanov's theorem to infer the existence of solutions.\medskip 

Consequently, there is, up to now, no systematic treatment of stochastic equations with arbitrary bounded, measurable 
drift. However, in the case where $E$ is a Hilbert space, there are several articles about equations with measurable 
drift under additional assumptions, typically a monotonicity assumption. Let us mention Da Prato and R\"ockner \cite{dpr02}, 
where the authors construct solutions as follows. First, they prove well-posedness of the associated Kolmogorov equation 
on a suitable $L^2$-space and then show that the associated semigroup is strongly Feller. Then they construct from these 
transition probabilities a Markov process in a canonical way which subsequently is shown to have a continuous modification.

A different approach is taken by Bogachev, Da Prato and R\"ockner in \cite{bdpr10, bdpr11}, where uniqueness of solutions 
is proven by establishing uniqueness for the associated Fokker-Planck equation, rather than for the Kolmogorov equation.
Existence of solutions is derived from Girsanov's Theorem in the case where $G$ is invertible but an existence result is 
missing in the general case.

Finally, let us mention the recent article by Da Prato, Flandoli, Priola and R\"ockner \cite{dpfpr11} where the authors, 
exploiting again Girsanov's theorem, establish pathwise uniqueness (rather than uniqueness in law) for equations with space-time white noise and 
bounded measurable drift.
\medskip 

The results obtained in this article can also be used to establish existence and uniqueness for equations with 
unbounded measurable drift under additional assumptions. This will be done elsewhere.\medskip 

Let us now describe our strategy for the proof of Theorem \ref{t.main}, this will also give an overview of this 
article.

Our main tool for the proof is a perturbation result (Theorem \ref{t.pert}) for strong Feller semigroups 
which may be of independent interest. This result is proved in Section \ref{sect.perturb} and is 
similar to a perturbation theorem for bi-continuous 
semigroups due to Farkas \cite{farkasSF}. The main novelty is that we do not perturb a semigroup on $C_b(E)$, the 
space of bounded, continuous functions on $E$, but on $B_b(E)$, the space of bounded measurable functions on $E$. 
Since the orbits of a bounded measurable function under a transition semigroup in general have 
no better regularity than some weak measurability, the proof of the perturbation result will be based on integration 
theory on norming dual pairs \cite{k10}. Working on $B_b(E)$ instead of $C_b(E)$ has two advantages: (i) we can 
allow perturbations taking values in $B_b(E)$ and (ii) we prove the strong Feller property of the 
perturbed semigroup along the way.

This perturbation result gives us a candidate $\scrP$ for the transition semigroup associated to 
equation $[A,F,G]$.\smallskip 

The actual proof of Theorem \ref{t.main} will be given in Section \ref{sect.well}. We first prove (Theorem \ref{t.uil})
that the distribution of a solution to \eqref{eq.semilinear} is uniquely determined by the perturbed semigroup $\scrP$ 
and the initial distribution of the solution. Thus uniqueness in law holds for \eqref{eq.semilinear}. Using that 
we know in advance the only possible transition semigroup for the equation, we prove in Theorem \ref{t.limit} 
that well-posedness of $[A,F,G]$ is stable under taking bounded and pointwise limits in $F$. This allows us 
to finish the proof of Theorem \ref{t.main} through a monotone class type argument.

\section{Preliminaries}\label{sect.prelim}

Throughout, $(M, d)$ is a complete, separable metric space. Its Borel $\sigma$-algebra is denoted by $\cB(M)$ and 
the spaces of scalar-valued Borel-measurable, bounded Borel-measurable, continuous 
and bounded continuous functions are denoted by $B(M), B_b(M), C(M)$ and $C_b(M)$ respectively. 
$\M (M)$ refers to the space of complex
measures on $(M, \cB (M))$ and  $\cP (M)$ to the subset of all probability measures. 
The symbol $\norm{}_\infty$ denotes the supremum norm and $\norm{}_\TV$ the total variation norm.

If $X$ and $Y$ are a Banach space, we write $\cL (X, Y)$ for the bounded linear operators 
from $X$ to $Y$. In the case $X = Y$ we write $\cL (X)$ shorthand for $\cL (X, X)$. If $\tau$ is a locally 
convex topology on $X$, then we write $\cL (X, \tau )$ for the algebra of $\tau$-continuous linear operators 
on $X$. 

\subsection{Kernel operators}
A \emph{kernel} on $(M, \cB (M))$ is a map $k : M\times \cB (M) \to \CC$ such that 
\begin{enumerate}
 \item $k(\cdot , A)$ is $\cB (M)$-measurable for all $A \in \cB (M)$;
 \item $k(x, \cdot ) \in \M (M)$ for all $x \in M$;
 \item $\sup_{x \in M} \|k(x, \cdot)\|_\TV < \infty$.
\end{enumerate}
If $k(x, \cdot ) \in \cP(M)$ for all $x \in M$, then $k$ is called a \emph{Markovian kernel}.

Given a kernel $k$ on $(M, \cB (M))$, we can define a bounded linear operator $\cT$ on $B_b(M)$ by
\begin{equation}\label{eq.kernelop}
  \cT f(x) := \int_M f(y)\, k(x, dy)\quad f \in B_b(M)\, , \, x \in M\,\, .
\end{equation}
An arbitrary operator $\cT \in \cL(B_b(M))$ is called \emph{kernel operator} if it is given in this way 
for some kernel $k$. In this case, $k$ is uniquely determined by $\cT$ and called the \emph{associated kernel}.

It is well-known that $\cT \in \cL (B_b(M))$ is a kernel operator if and only if
 $\cT f_n$ converges pointwise to $\cT f$ whenever $(f_n) \subset B_b(M)$ is a bounded sequence which converges 
pointwise to $f \in B_b(E)$. Yet another characterization of kernel operators can be given using 
the \emph{weak topology} $\sigma := \sigma (B_b(M), \M (M))$. $\cT \in \cL (B_b(M))$ is a kernel operator
if and only if $\cT \in \cL (B_b(M), \sigma )$, see \cite[Proposition 3.5]{k10}.  Moreover, 
$\cL (B_b(M), \sigma ) \subset \cL (B_b(M))$, i.e.\ a $\sigma$-continuous operator is necessarily bounded.
Note that a sequence $(f_n) \subset B_b(M)$ converges to $f$ with respect to $\sigma$ if and only if the 
sequence is bounded and converges pointwise. We write $\weak$ to indicate convergence with respect to $\sigma$.

For $\cT \in \cL (B_b(M), \sigma )$ we denote its 
$\sigma$-adjoint by $\cT'$. We note that $\cT'$ is the restriction of the norm adjoint $\cT^*$ to 
$\M (M)$; in fact, an operator $\cT \in \cL (B_b(M))$ is a kernel operator if and only if 
its norm adjoint leaves $\M (M)$ invariant.  We also note that $\cT'$ is an element of $\cL (\M (M))$ and we have
\[ \big( \cT'\mu )(A) = \int_M k(x, A)\, d\mu (x) \quad \forall\, A \in \B (M)\, ,\]
where $k$ is the kernel associated with $\cT$.

A kernel operator $\cT$ is called \emph{Markovian} if its associated kernel is Markovian; it is 
called a \emph{strong Feller operator} if $\cT B_b(M) \subset C_b(M)$. 

\subsection{Very weak integration}
Simple examples show that for integration of $B_b(M)$ or $\M (M)$-valued functions, the notion 
of Bochner integrability and even that of Pettis integrability is often too strong. We will therefore use
a weak notion of integrability studied in \cite{k10}.

A \emph{norming dual pair} is a pair $(X, Y)$ where $X$ is a Banach space and $Y$ is a norm-closed 
subspace of the dual $X^*$ which is norming for $X$, i.e.\ $\|x\| = \sup\{ |\dual{x}{y}|\, : \, 
y \in Y\, ,\, \|y\|\leq 1\}$.
We are interested in the situations where $X = B_b(M)$ or $X = C_b(M)$ and $Y = \M (M)$ and in the 
situation where $X = \M (M)$ and $Y = B_b(M)$.

Now let $(S, \mathscr{S}, m)$ be a $\sigma$-finite measure space. A function $\Phi : S \to X$ is called
\emph{scalarly $Y$-measurable} (\emph{scalarly $Y$-integrable}) if $\dual{\Phi}{y}$ is measurable (integrable) 
for all $y \in Y$. The function $\Phi$ is called $Y$-integrable if for all $A \in \mathscr{S}$ there exists an element 
$x_A \in X$ such that 
\[ \dual{x_A}{y} = \int_S \one_A(s)\dual{\Phi (s)}{y}\, dm (s) \quad \forall\, y \in Y\, . \]
In this case, $\int_A \Phi (s)\, dm (s) := x_A$ is called the \emph{$Y$-integral of $\Phi$ over $A$}.

If $(X, Y) = (C_b(M), \M (M))$ or $(X, Y) = (\M (M), B_b(M))$, then 
every scalarly $Y$-measurable function $\Phi$ such that $\|\Phi\|$ is majorized by a function in $L^1(S, dm )$
is $Y$-integrable, cf.\ \cite[Section 6]{k10}. 

\subsection{Strong Feller semigroups}

\begin{defn}\label{d.sf}
A family $\cT := (\cT (t))_{t\geq 0} \subset \cL (B_b(M), \sigma )$ is called \emph{strong Feller semigroup} if
\begin{enumerate}
 \item $\cT (0) = I$ and $\cT (t+s) = \cT (t) \cT(s)$ for all $s,t\geq 0$;
 \item For every $t>0$, $\cT (t)$ is a strong Feller operator;
 \item For all $f \in C_b(M)$ we have $\cT (t) f \weak  f$  as $t \downarrow 0$.
\end{enumerate}
$\cT$ is called \emph{Markovian} if $\cT (t)$ is Markovian for all $t\geq 0$. A Markovian strong Feller semigroup 
is called \emph{irreducible} if $\cT (t)\one_U(x)>0$ for all $x \in M$, $t>0$ and $U \subset M$ open.
\end{defn}

By \cite[Lemma 5.9]{k10}, assumptions (1) and (3) in the definition imply that $\cT$ is exponentially 
bounded, i.e.\ there exist constants $C \geq 1$ and $\omega \in \CR$ such that
$\|\cT(t)\|\leq Ce^{\omega t}$ for all $t\geq 0$; we will say that $\cT$ is of \emph{type} $(C, \omega )$.\

Given exponential boundedness, (3) is equivalent to $\cT (t)f \to f$ pointwise as $t\downarrow 0$, i.e.\ 
$\cT$ is \emph{stochastically continuous}. This assumption excludes pathological examples such as 
$\cT (0) = I$ and $\cT (t) = 0$ for all $t>0$.\smallskip 

By the results of \cite[Section 6]{k10}, $\cT$ is an integrable semigroup. This means that there 
exists a family $\cR = (\cR (\lambda ))_{\Re\lambda > \omega} \subset \cL (B_b(M), \sigma )$ such that
\[ \dual{\cR (\lambda )f}{\mu} = \int_0^\infty e^{-\lambda t}\dual{\cT (t)f}{\mu}\, dt \]
for all $f \in B_b(M), \mu \in \M (M)$ and $\lambda \in \CC$ with $\Re\lambda > \omega$.

Note that this is equivalent to the statement that $\cR (\lambda )f$ is the $\M (M)$-integral of 
$e^{-\lambda \cdot} \cT (\cdot )f$ over $(0,\infty )$ and $\cR (\lambda )'\mu$ is the 
$B_b(M)$-integral of $e^{-\lambda \cdot } \cT (\cdot )' \mu$ over $(0,\infty )$.\medskip 

$\cR$ is called the \emph{Laplace transform} of $\cT$. It was proved in \cite{k10} that $\cR$ is always 
a pseudo-resolvent, i.e.\ for $\lambda_1, \lambda_2$ with $\Re\lambda_1, \Re\lambda_2 > \omega$ we have
\[ (\lambda_1-\lambda_2)\cR(\lambda_2)\cR(\lambda_1) = \cR (\lambda_2) -\cR (\lambda_1)\, .\]
It follows that there exists a unique multi-valued operator $\cL \subset B_b(M) \times B_b(M)$ such that 
$\cR (\lambda ) =
(\lambda -\cL )^{-1}$ for all $\Re\lambda > \omega$, see \cite[Proposition A.2.4]{haase}. 
This equation is understood in the graph-sense, i.e.\ 
\[ \big\{ (f, \cR(\lambda )f )\, :\, f \in B_b(M)\big\} =
 \big\{ (f,g)\, : \, (g, \lambda g -f) \in \cL \big\}\, .
\]
For more information on pseudo-resolvents and multi-valued operators we refer to \cite[Appendix A]{haase}.

The operator $\cL$ will be called the \emph{full generator} of $\cT$. By \cite[Proposition 5.7]{k10}, 
$(f,g) \in \cL$ if and only if $\int_0^t \cT (s)g\, ds = \cT (t)f - f$, where the integral is an $\M (M)$-integral.
It follows that $\cL$ is the full generator in the sense of \cite{ek}.
\medskip 

We should mention that, in general, the operator $\cL$ is indeed multivalued. However, restricting the semigroup
to $C_b(M)$,  we can associate a single valued generator with $\cT|_{C_b(M)}$. In the language of norming dual pairs, 
we replace $(B_b(M), \M (M))$ by $(C_b(M), \M (M))$. It is easy to see that the generator of $\cT|_{C_b(M)}$, 
defined as above  via the Laplace transform, is exactly the part $\cL|_{C_b(M)} := \cL\cap (C_b(M)\times C_b(M))$
 of $\cL$ in $C_b(M)$. 
Since $\cT|_{C_b(M)}$ is $\sigma$-continuous at 0, it follows from \cite[Theorem 2.10]{k09} \
that $\cL|_{C_b(M)}$ is single-valued and sequentially $\sigma$-densely defined in $C_b(M)$; 
furthermore, $\cL|_{C_b(M)}$ is exactly the derivative with respect to $\sigma$ of $\cT$ at $0$.

We will also need some properties of the generator of $\cT'$. As we have noted above, the Laplace transform of 
$\cT'$ is $\cR'$. It is not hard to see that $\cR'$ is the resolvent of $\cL'$, the $\sigma$-adjoint 
of $\cL$.
Since $\cT'$ is $\sigma (\M (M), C_b(M))$-continuous at $0$, it follows similarly as above 
that $\cL'$ is single-valued and sequentially $\sigma (\M (M), C_b(M))$ densely defined.

\subsection{Cores}
Of particular importance for us will be the concept of a \emph{core} of the full generator. In our case, there 
are some subtleties due to the fact that we deal with multi-valued operators.
\begin{defn}
Let $\cT$ be a strong Feller semigroup with full generator $\cL$. A subset $\cD$ of $\cL$ is called 
a \emph{core} for $\cL$ if $\overline{\cD}^{\, \sigma\times \sigma} = \cL$.
\end{defn}
We now extend a well-known result from the theory of strongly continuous semigroups to our more general setting.

We note that if $D \subset D(\cL)$ is a subspace of $B_b(M)$ which is invariant under the semigroup 
$\cT$, then $\cT$ leaves the norm closure $\overline{D}$ invariant. Furthermore, $\cT|_{\overline{D}}$ is strongly 
continuous. This is an easy consequence of the fact that $t \mapsto \cT (t)f$ is $\|\cdot\|_\infty$-continuous 
for all $f \in D(\cL)$, see \cite[Remark 2.5]{k09}. 
If $L$ denotes the generator of $\cT|_{\overline{D}}$, then $(f, Lf) \in \cL$ for 
all $f \in D(L)$.

\begin{prop}\label{p.core}
Let $\cT$ be a strong Feller semigroup with full generator $\cL$ and $D$ be a subspace of $D(\cL)$ which is 
$\sigma$-dense in $B_b(M)$ and invariant under $\cT$. Then $\mathscr{D} := \{ (f, Lf)\, : f \in D\}$, with 
$L$ as above, is a core for $\cL$.
\end{prop}

\begin{proof}
Suppose that there is some $(f_0, g_0) \in \cL$ which does not belong to the $\sigma\times \sigma$-closure of $\cD$.
By the Hahn-Banach theorem, applied in the locally convex space $(B_b(M)\times B_b(M), \sigma\times\sigma)$, 
there exists $(\mu, \nu ) \in \M (M)\times \M (M)$ such that
\begin{eqnarray}
 0 & = & \dual{f}{\mu} + \dual{Lf}{\nu}\quad \forall\, f \in D \label{eq.1}\\
0 & \neq & \dual{f_0}{\mu} + \dual{g_0}{\nu}\label{eq.2} .
\end{eqnarray}
Since $\cT D \subset D$, equation \eqref{eq.1} yields for $f \in D$ and $t\geq 0$ 
\[ \dual{\cT (t)f}{\mu} = - \dual{L \cT (t)f}{\nu}\, .\]
Multiplying with $e^{-\lambda t}$ for $\lambda$ large enough and integrating over $(0,\infty )$, it follows that
\[ \dual{R(\lambda, L) f}{\mu} = - \dual{L R(\lambda, L) f}{\nu} = - \dual{\lambda R(\lambda, L)f -f}{\nu}\, , \]
where $R(\lambda, L)$ is the resolvent of $L$. Since $R(\lambda, L) = \cR (\lambda )|_{\overline{D}}$, we have
\[ \dual{\cR (\lambda )f}{\mu} = \dual{f-\lambda \cR(\lambda )f}{\nu} \]
for all $f \in D$.
By the $\sigma$-density of $D$ and the $\sigma$-continuity of $\cR (\lambda )$, this 
equality remains valid for arbitrary $f \in B_b(M)$.

Now observe that $\cR (\lambda ) (\lambda f_0 - g_0) = f_0$. Using the above equation for $f = \lambda f_0 - g_0$, we obtain 
$\dual{f_0}{\mu} = \dual{g_0}{\nu}$ --- a contradiction to \eqref{eq.2}.
\end{proof}
 
\section{Perturbation of Strong Feller semigroups}\label{sect.perturb}

In our perturbation theorem, we will make the following assumptions:

\begin{hyp}\label{hyp1}
Let $\cT = (\cT(t))_{t\geq 0}$ be a strong Feller semigroup with full generator $\cL$. 
We assume that $\cT$ is of type $(C, \omega )$ and denote the Laplace transform of $\cT$ by 
$\cR = (\cR (\lambda ))_{\Re\lambda > \omega }$. Furthermore, let $B : D(B) \to B_b(M)$ 
be a linear operator with $D(\cL ) \subset D(B)$, enjoying the following properties
\begin{enumerate}
 \item $B\cR (\lambda ) \in \cL (B_b(M), \sigma )$ for one (equivalently, all) $\Re\lambda > \omega$;
 \item for every $t >0$ the map $B \cT (t)$, initially defined on $D(\cL)$, has a (necessarily unique) extension
to an operator in $\cL (B_b(M), \sigma )$; by slight abuse of notation, we will denote this extension still 
by $B\cT (t)$;
 \item $s \mapsto B\cT(s)f$ is scalarly $\M (M)$-measurable for all $f \in B_b(M)$;
 \item There is a function $\varphi \in L^1_\mathrm{loc}([0,\infty)$ such that $\|B\cT(s)\| \leq \varphi (s)$ for all 
$s >0$.
\end{enumerate}
\end{hyp}

\begin{rem}\label{rem1}
Requirement (1) is equivalent to the statement that if $(f_\alpha, g_\alpha )$ is a net in $\cL$ such that 
$f_\alpha \weak f, g_\alpha \weak  g$, then $Bf_\alpha \weak Bf$. 
\end{rem}

\begin{thm}\label{t.pert}
Assume Hypothesis \ref{hyp1}. Then $\cL + B$ is the full generator of a strong Feller semigroup 
$\scrP = (\scrP (t))_{t\geq 0}$ which satisfies the integral equation
\begin{equation}\label{eq.inteq}
 \scrP (t) f = \cT (t)f + \int_0^t \scrP (t-s)B\cT(s)f\, ds \quad f \in B_b(M)\,.
\end{equation}
Here the integral is understood as a $\M (M)$-integral.
\end{thm}

\begin{example}
(Adding a potential term)

If $\cT$ is a strong Feller semigroup and $B \in \cL(B_b(M), \sigma )$, then $\cT$ together with 
$B$ satisfies Hypothesis \ref{hyp1}. Indeed, $B\cT(s), B\cR(\lambda ) \in \cL(B_b(M), \sigma )$ as composition of two 
kernel operators. Since $\dual{B\cT(s)f}{\mu} = \dual{\cT(s)}{B'\mu}$, the scalar $\M (M)$-measurability 
of $s \mapsto BT(s)f$ follows from that of $s \mapsto T(s)f$. Condition (4) in Hypothesis 
\ref{hyp1} follows from the boundedness of $B$ and the exponential boundedness of $\cT$.

As an example for $B$, let $V \in B_b(M)$ and put $Bf(x) = V(x)f(x)$. Then 
$B \in L(B_b(M), \sigma )$ as it is associated with the kernel $k$, given by 
$k(x, \cdot ) = V(x)\delta_x(\cdot)$, where $\delta_x$ is the Dirac measure in $x$.
\end{example}

\subsection{Proof of the perturbation theorem}

Similar to the perturbation result of \cite{farkasSF}, and also similar to the proof of 
the classical Miyadera-Voigt perturbation theorem, the proof of Theorem \ref{t.pert} depends 
on a fixed point argument. We begin by introducing a Banach space $\fX$ and an operator $\fV : \fX \to \fX$
such that the perturbed semigroup $\scrP$ is given as a fixed point of $\fV$.

\begin{defn}\label{def.1}
Let $\cT$ be a strong Feller semigroup with full generator $\cL$ and $t>0$.
We define $\fX ([0,t])$ as the space of all functions $\cF: [0,t] \to \cL(B_b(M), \sigma)$ such that
\begin{enumerate}
 \item $\sup_{s \in [0,t]}\|\cF(s)\| <\infty$;
 \item $\cF(s)$ is a strong Feller operator for all $s \in (0, t]$;
 \item for every $\mu \in D(\cL')$ the map $[0,t] \ni s \mapsto \cF(s)'\mu$ is right-continuous in 
the total variation norm.
\end{enumerate}
\end{defn}

Using that the limit of strong Feller operators in the operator norm is again a strong Feller operator, it is 
easy to see that $\fX ([0,t])$ is a Banach space for the norm $\|\cdot\|_\fX$ defined by 
$\|\cF\|_\fX := \sup_{s \in [0,t]}\|\cF(s)\|_{L(B_b(E))}$.  

Under Hypothesis \ref{hyp1}, we define for $\cF \in \fX ([0,t])$ 
\begin{equation}\label{eq.volterra}
[\fV_t \cF](s)f := \int_0^s \cF(s-r)B\cT(r)f\, dr\, , \quad f \in B_b(E)\, ,\, s \in [0,t]\, ,
\end{equation}
where the integral is understood as a $\mathcal{M}(M)$-integral. The proof that this integral exists is 
part of the following

\begin{lem}\label{l.voltop}
Under Hypothesis \ref{hyp1}, $\fV_t$ defines a bounded linear operator from $\fX ([0,t])$ into itself with
$\|\fV_t\| \leq \int_0^t\varphi (s) ds$.
\end{lem}

\begin{proof}
We first prove that for $\cF \in \fX ([0,t]), f \in B_b(E)$ and $s \in [0,t]$ the 
integral in \eqref{eq.volterra} exists as an $\M (M)$-integral.

Let $\mu \in D(\cL')$. Then $r \mapsto \dual{\cF(s-r)B\cT(r)f}{\mu}$ is measurable. To see this, 
observe that since $r \mapsto \cF (r)'\mu$ is right-continuous with respect to the total variation norm, 
it is strongly measurable. Hence, there exists a sequence $(\Phi_k)_{k\in\CN}$ of simple functions, 
say $\Phi_k(r) = \sum_{j=1}^{n_k} \one_{A_{jk}}(r)\nu_{jk}$, which converges to $\cF (s-\cdot )'\mu$, pointwise
in the total variation norm. By (3) in Hypothesis \ref{hyp1},  $\dual{B\cT (r)}{\nu_{jk}}$ is a measurable 
function of $r$ for all $k \in \CN$ and $j=1, \ldots, n_k$. Since $\Phi_k$ converges in norm 
to $\cF (s-\cdot )'\mu$, it follows that for $r \in (0,s)$
\[ \dual{\cF (s-r)B\cT (r)f}{\mu} = \lim_{k\to \infty} \sum_{j=1}^{n_k}\one_{A_{jk}}(r)\dual{B\cT (r)}{\nu_{jk}}\, , \]
proving the claimed measurability.

Now let $\mu \in \M (M)$ be arbitrary. By the sequential $\sigma (\M (M), C_b(M))$-density of $D(\cL')$, there 
exists a sequence $(\mu_n) \subset D(\cL')$ which converges to 
$\mu$ with respect to $\sigma (\M(M), C_b(M))$. Note that since $\cF (r)$ is a strong Feller operator for 
$r \in (0,t]$, we have $F(s-r)BT(r)f \in C_b(E)$ for all $r \in [0,s)$. Consequently,
\[ \dual{\cF (s-r)B\cT(r)f}{\mu} = \lim_{n\to \infty}\dual{\cF (s-r)B\cT(r)f}{\mu_n} \quad 
 \forall\, r \in [0,s)\, .
\]
This proves that $r \mapsto \cF (s-r)B\cT(r)f$ is scalarly $\M (M)$-measurable.

Next observe that $\|\cF (s-r)B\cT (r)f\|_\infty \leq  \varphi (r) \|f\|_\infty\|\cF\|_\fX$. The results of 
\cite{k10} imply that the $C_b(M)$-valued function $\cF (s-\cdot )B\cT (\cdot )f$ is $\M (M)$-integrable; in 
particular, its integral over $(0,s)$ defines an element of $C_b(M)$. 
Furthermore,  $\|\fV_t\| \leq \sup_{s \in (0,t)} \int_0^s \varphi (r)\, dr =
\int_0^t \varphi (r)\, dr$ as claimed follows once we prove that $\fV_t$ maps $\fX ([0,t])$ into itself.\medskip 

It remains to prove that $\fV_t\cF \in \fX ([0,t])$ for all $\cF \in  \fX ([0,t])$.

By the above, $f \mapsto [\fV_t \cF](s)f$ defines a bounded linear operator from 
$B_b(M)$ to $C_b(M)$. Let $f_n \weak f$. 
Since $B\cT(r), \cF (s-r) \in \cL (B_b(M), \sigma )$, the dominated convergence theorem yields
for $x \in M$
\[ 
\begin{aligned}
\big([\fV_t\cF](s)f_n\big)(x) = & \int_0^s \dual{\cF (s -r)B\cT(r)f_n}{\delta_x}\, dr\\ 
& \to 
 \int_0^s \dual{\cF (s -r)B\cT(r)f}{\delta_x}\, dr = \big([\fV_t\cF](s)f)(x) \,.
\end{aligned}
\]
This shows that $[\fV_t F](s)$ is sequentially $\sigma$-continuous and hence a kernel operator.\smallskip

Finally, let $0\leq s_1<s_2\leq t$ and 
$f \in B_b(E), \mu \in D(\cL')$ be given. We have
\begin{eqnarray*}
&& | \dual{[\fV_t \cF](s_2)f - [\fV_t \cF] (s_1)f}{\mu}|\\
& \leq & \int_{s_1}^{s_2}|\dual{\cF(s_2-r)B\cT(r)f}{\mu}|\, dr \\
& & \quad + \Big| \int_0^{s_1}\langle B\cT(r)f\, dr, \cF(s_2-r)'\mu - \cF(s_1-r)'\mu\rangle\, dr\Big|\\
& \leq & \|\cF\|_\fX \int_{s_1}^{s_2}\varphi (r) \|f\|_\infty\|\mu\|_\TV\, dr\\
&&\quad
+ \|f\|_\infty \int_0^{s_1} \varphi (r) \|\cF(s_2-r)\mu -\cF(s_1-r)\mu\|_\TV\, dr\, .  
\end{eqnarray*}
Taking the supremum over $f \in B_b(E)$ with $\|f\|\leq 1$, we find
\[ 
\begin{aligned}
\| [\fV_t \cF](s_2)'\mu & - [\fV_t \cF](s_1)'\mu \|_\TV 
 \leq \|\mu\|_\TV \|\cF\|_\fX \int_{s_1}^{s_2}\varphi (r)\, dr\\ + 
&\int_0^{s_1} \varphi (r) \|\cF(s_2 -r)\mu -\cF(s_1-r)\mu\|_\TV\, dr \, .
\end{aligned}
\]
Clearly, $\int_{s_1}^{s_2}\varphi (r)\, dr \to 0$ as $s_2 \downarrow s_1$. 
Since $\mu \in D(\cL')$, the function $\cF(\cdot )'\mu$ is right-continuous in the total variation 
norm. Thus the integrand in the second integral converges pointwise to 0. By dominated convergence, 
also this second integral converges to 0 and it follows that $s \mapsto [\fV_t \cF](s)'\mu$ is right-continuous 
in the total variation norm for $\mu \in D(\cL')$.
\end{proof}

From the norm estimate $\|\fV_t\| \leq \int_0^t\varphi (r)\, dr$, we immediately obtain:

\begin{cor}\label{c.res}
$\fV_t$ converges to $0$ in the operator norm. In particular, $1 \in \rho (\fV_t)$ for $t$ small enough. 
\end{cor}

We now study the interaction of $B$ with the Laplace transform $\cR = (\cR (\lambda ))_{\Re\lambda > \omega}$.

\begin{lem}\label{l.BR}
Assume Hypothesis \ref{hyp1}. Then
\begin{enumerate}
 \item For $\Re \lambda > \max\{\omega, 0\}$, we have
\begin{equation}\label{eq.BRint}
  B\cR(\lambda)f = \int_0^\infty e^{-\lambda t}B\cT(t)f\, dt\quad f \in B_b(M)\, ,
\end{equation}
where the integral exists as a $\M (M)$-integral.
\item For $\Re\lambda$ large enough, $\|B\cR(\lambda)\| < 1$.
\item For $\Re\lambda$ large enough, $\lambda \in \rho (\cL + B)$ and
\begin{equation}\label{eq.resolvent}
  (\lambda -\cL - B)^{-1} = \cR(\lambda)(I - B\cR(\lambda))^{-1}\,\, .
\end{equation}
\end{enumerate}
\end{lem}

\begin{proof}
(1) Let $\eps >0$ and $\Re\lambda > \max\{\omega, 0\}$. Then $B\cT (\eps )\cR (\lambda ) \in \cL (B_b(M), \sigma )$ as 
composition of the kernel operators $B\cT (\eps)$ and $\cR(\lambda )$.
Furthermore, since $B \cT (\eps ) \in \cL (B_b(M), \sigma )$, 
\[ 
\begin{aligned}
B\cT (\eps )\cR(\lambda )f = & \,B \cT (\eps )\int_0^\infty e^{-\lambda t}\cT (t)f \, dt 
 = \int_0^\infty e^{-\lambda t}B\cT (t+\eps)f \, ds\\
 = & \,
 \int_\eps^\infty e^{\lambda \eps} e^{-\lambda s} B\cT (s )f\, ds\, .
\end{aligned}
\]
We note that $\one_{(\eps, \infty)}(s)e^{\lambda \eps} e^{-\lambda s} B\cT(s)f$ converges in norm 
to $e^{-\lambda s}B\cT(s)f$ as $\eps \to 0$, for all $s>0$. Furthermore, 
with $c_1 := \max\{e^{\Re\lambda}, 1\}$, we have
\[ 
 \one_{(\eps, \infty)}\|e^{\lambda \eps} e^{-\lambda s} B\cT (s)f\| 
\leq c_1 e^{-\Re\lambda s} \cdot \big[ \one_{(0,1)}(s)\varphi (s) + \varphi (1)Ce^{\omega (s-1)}\big] 
\cdot \|f\|_\infty
\]
for $\eps \in (0,1)$, 
and the function on the right-hand side is integrable on $(0,\infty)$. It follows from dominated convergence, cf.\ 
\cite[Lemma 4.7]{k10}, that $s \mapsto e^{-\lambda s}B\cT (s)f$ is $\M (M)$-integrable on $(0,\infty )$ and 
\[ \lim_{\eps \to 0} \int_\eps^\infty e^{\lambda \eps}e^{-\lambda s}B\cT (s)f\, ds 
 = \int_0^\infty e^{-\lambda s}B\cT (s)f\, ds 
\]
in norm. Noting that $\cL|_{C_b(M)}$ is single-valued, we put $f_0 = (\lambda - \cL|_{C_b(M)})\cR (\lambda )f 
\in C_b(M)$. Then
\[ \cT (\eps )\cR(\lambda )f = \cT (\eps )\cR(\lambda )f_0 = \cR (\lambda )\cT (\eps )f_0 
 \weak \cR(\lambda )f_0 = \cR (\lambda )f\, ,
\]
as $\eps \to 0$. Using (1) of Hypothesis \ref{hyp1}, (1) follows.\medskip 

(2) Let $\delta >0$ and $\Re\lambda > \max\{\omega, 0\}$. Then, for $t>0$, we have
\[ \|e^{-\lambda t}B \cT(t)\| \leq \one_{(0,\delta )}(t)\varphi (t) + \one_{[\delta, \infty )} 
 C\varphi (\delta )e^{-\Re\lambda t } e^{\omega (t-\delta )}\, .
\]
By the representation from (1), we find
\[ \|B\cR (\lambda )\| \leq \int_0^\delta \varphi (t)\, dt + Ce^{-\Re\lambda \delta} \varphi (\delta)\int_0^\infty 
 e^{(\omega - \Re\lambda )t}\, dt\, .
\]
Given $\eps>0$, we see that $\|B\cR(\lambda )\| \leq \eps$ if we first choose $\delta$ small enough such that
$\int_0^\delta \varphi (t)\, dt \leq \eps/2$ and then $\Re\lambda$ large enough, such that the second term is also 
less than $\eps/2$.\medskip 

(3) For $\Re\lambda > \omega$ we have
\[ \lambda -\cL - B = (I - B\cR (\lambda ) )(\lambda - \cL)\, .\]
Indeed, since $I - B\cR (\lambda )$ is a bounded operator, both sides have the same domain. Now let $f \in D (\cL)$. If 
$g \in (\lambda - \cL - B)f$, then there exists $h \in (\lambda - \cL )f$ with $g = h - Bf$. We note that
$h \in (\lambda - \cL )f$ iff $\cR(\lambda )h = f$. It follows that
\[ g = h - Bf = (I - B\cR(\lambda ))h \in (I - B\cR (\lambda ))(\lambda - \cL )f\, .\]
Conversely, assume that $g \in (I - B\cR (\lambda ))(\lambda - \cL)f$. Then there exists $h \in (\lambda - \cL )f$, 
i.e.\ $\cR (\lambda )h = f$, with $g = (I - B\cR (\lambda )h = h - Bf$. Thus $g \in (\lambda - \cL - B)f$.

If we choose $\Re\lambda$ large enough such that $\|B\cR (\lambda )\|<1$, then $I - B\cR (\lambda )$ is invertible
in $\cL (B_b(M), \sigma )$ and \eqref{eq.resolvent} follows.
\end{proof}

\begin{proof}[Proof of Theorem \ref{t.pert}]

First note that for every $t_0 >0$ we have $\cT \in \fX ([0,t_0])$. Indeed, conditions (1) and (2) 
in Definition \ref{def.1} are obvious and (3) holds since $t \mapsto \cT(t)'\mu$ is $\|\cdot\|_{\mathrm{TV}}$-continuous
for $\mu \in D(\cL')$, cf.\ \cite[Remark 2.5]{k09}. Now pick $t_0>0$ such that $1 \in \rho (\fV_{t_0})$.
For $0\leq t = k t_0 + t_1$, where $k \in \CN$ and $t_1 \in [0, t_0)$, we define
\[ \scrP(t) := \big( \big[ R(1, \fV_{t_0}) \cT|_{[0,t_0]}\big](t_0)\big)^k\big[ R(1,\fV_{t_0})\cT|_{[0,t_0]}\big](t_1)\, .\]
By the definition of $\fX ([0,t_0])$ we have $\scrP := (\scrP(t))_{t\geq 0} \subset \cL(B_b(E), \sigma )$ and 
$\scrP(t)$ has the strong Feller property for all $t>0$. We also note that \eqref{eq.inteq} is 
satisfied for $t\leq t_0$. By Corollary \ref{c.res} we have $[\fV_{t_0} \scrP](t) \to 0$ in the operator 
norm as $t \downarrow 0$. Hence, by equation \eqref{eq.inteq}, $\scrP(t)f \weak f$ as $t\downarrow 0$.

The proof that $\scrP$ is a semigroup and satisfies \eqref{eq.inteq} for all $t>0$  is basically algebraic and 
follows the lines of the proof in \cite{farkasSF}. 
\medskip 

We now identify the full generator of $\scrP$. By Lemma \ref{l.BR}, we may choose $\lambda \in \CR$ so large that 
$\|B\cR(\lambda )\| <1$. Let us denote the generator of $\scrP$ by $\mathscr{S}$, i.e.\ $\mathscr{S}$ is 
the unique multivalued operator such that $(\lambda - \mathscr{S})^{-1} = \mathscr{Q}(\lambda )$ for 
$\Re\lambda > \omega$, where $(\mathscr{Q}(\lambda ))_{\Re\lambda > \omega}$ is the Laplace transform of 
$\scrP$.

By \eqref{eq.inteq} and Fubini's theorem, 
\begin{eqnarray*}
\mathscr{Q}(\lambda ) f & = &  \int_0^\infty e^{-\lambda t}\Big( \cT(t)f + \int_0^t\scrP(t-s)B\cT(s)\, ds\Big)\, dt\\
& = & \cR (\lambda ) f + \int_0^\infty e^{-\lambda s} \int_0^\infty \scrP(r)B\cT(s)f \, dr ds\\
& = & \cR(\lambda)f + \mathscr{Q}(\lambda ) B\cR(\lambda )f\, .
\end{eqnarray*}
This implies that 
\[ \mathscr{Q}(\lambda ) = \cR(\lambda )(I- B\cR(\lambda ))^{-1}\, .\]
By Lemma \ref{l.BR} (3),  $\mathscr{S} = \cL + B$.
\end{proof}

\subsection{Continuity properties of the perturbed semigroup}
It follows immediately from equation \eqref{eq.inteq} and Corollary \ref{c.res} that the perturbed 
semigroup $\scrP$ enjoys the same continuity properties at $0$ as $\cT$. In many cases, the semigroup $\cT$ is known to have 
better continuity properties than the stochastic continuity assumed in Definition \ref{d.sf}, namely, for every 
$f \in C_b(M)$ we have $\cT (t )f \to f$, uniformly on the compact subsets of $M$ as $t\downarrow 0$, i.e.\ $\cT (t)f
\to f$ with respect to $\tau_{\mathrm{co}}$, the compact-open topology. In fact, in virtually all known cases, 
for $f \in C_b(M)$ the whole orbit $t \mapsto \cT (t)f$ is $\tau_{\mathrm{co}}$-continuous.

It is thus natural to ask whether in this situation the same is true for the perturbed semigroup $\scrP$.\medskip 

Before tackling this question, let us first reformulate the $\tau_{\mathrm{co}}$-continuity.
The \emph{strict topology} $\beta_0(M)$ (sometimes also called the \emph{mixed topology}) is defined as follows.

Let $\F_0(M)$ be the space of all scalar-valued functions $\psi$ which vanish 
at infinity, i.e.\ given $\eps>0$, there exists a compact set $K$ such that $|\psi (x)| \leq \eps$ for all 
$x \in M\setminus K$. For $\psi \in \F_0(M)$, the seminorm $p_\psi$ is defined by 
$p_\psi (f) = \|\psi f\|_\infty$. The strict topology $\beta_0(M)$ is the locally convex topology on $C_b(M)$ 
generated by the seminorms $(p_\psi )_{\psi \in \cF_0}$.

It is known that on $\|\cdot\|_\infty$-bounded sets, $\beta_0(M)$ coincides with the compact-open topology $\tau_\mathrm{co}$.
Hence, in view of exponential boundedness, $\cT|_{C_b(M)}$ has $\beta_0(M)$-continuous orbits if and only if 
it has $\tau_\mathrm{co}$-continuous orbits. Moreover, $\beta_0(M)$ is the strong Mackey topology of the 
dual pair $(C_b(M), \M (M))$, i.e.\ $\beta_0$ is the topology of uniform convergence on the $\sigma (\M (M), C_b(M))$-compact
subsets of $\M (M)$. In view of Prokhorov's Theorem, $\beta_0(M)$ is the topology of uniform convergence on 
the tight subsets of $\M (M)$. For the proof of these facts we refer to \cite{sentilles}, see also \cite{k09} for 
further references. 

\begin{prop}\label{p.cbcont}
Let $\cT$ be a strong Feller semigroup such that for every $f \in C_b(M)$ we have $\cT (t)f \to f$ with respect  
$\beta_0(M)$ as $t\downarrow 0$. Then $t \mapsto \cT (t)f$ is $\beta_0(M)$-continuous for every $f \in C_b(M)$.
\end{prop}

\begin{proof}
To begin with, we note that for $t>0$ the operator $\cT (t)$ is the product of two strong Feller operators and hence 
is ultra-Feller, i.e.\ the map $x \mapsto k_t(x, \cdot )$ is $d$ to $\|\cdot\|_\mathrm{TV}$-continuous, where 
$k_t$ is the kernel associated with $\cT (t)$. This in turn is equivalent to saying that $\cT (t)$ maps 
norm-bounded sets into relatively $\beta_0(M)$-compact sets. 
For Markovian operators, the proof of these facts can be found in \cite[Chapter 1, \S 5]{revuz}. The proof generalizes 
easily to general kernel operators.

By \cite[Proposition 2.12]{k09}, to prove that $t \mapsto \cT (t)f$ is $\beta_0(M)$-continuous for every $f \in C_b(M)$,
it suffices to prove that for every $f \in C_b(M)$ the set $\{ \cT (t)f \, : \, t \in [0,1]\}$ is relatively countably
$\beta_0(M)$-compact. Let a sequence $(t_n) \subset [0,1]$ be given. If $(t_n)$ has a subsequence converging to 
0, then $\cT (t_n)f$ has the $\beta_0(M)$-accumulation point $f$, since $\cT (t)f \to f$ with 
respect to $\beta_0(M)$ by assumption.

Now suppose that $t_n \geq \eps >0$ for all $n \in \CN$. In this case, we can write 
$\cT (t_n)f = \cT (\eps) \cT (t_n - \eps)f$. Since $\{\cT (t_n-\eps )f : n \in \CN\}$ is bounded, it is mapped to 
a relatively $\beta_0(M)$-compact set by the ultra-Feller operator $\cT (\eps)$. Hence also in this case $\cT(t_n)f$
has a $\beta_0(M)$-accumulation point. 
\end{proof}

Together with Corollary \ref{c.res}, we obtain

\begin{cor}
Under the assumption and with the notation of Theorem \ref{t.pert}, if $\cT|_{C_b(M)}$ has $\beta_0(M)$-continuous
orbits, then so does $\scrP|_{C_b(M)}$. 
\end{cor}

\subsection{Main application}
We now consider the situation of Theorem \ref{t.main}. Here $M$ is the real, separable Banach space $E$. We use 
Theorem \ref{t.pert} to perturb Ornstein-Uhlenbeck semigroups, i.e.\ the transition semigroup of 
equation $[A,0,G]$. For bounded $F$, we want that the perturbed semigroup is the transition semigroup of 
equation $[A,F,G]$. Hence, formally, the operator $B$ should be given by
\begin{equation}\label{eq.b}
 D(B) = C^1_b(E) \quad Bf (x) = \dual{F(x)}{D f(x)}\, ,
\end{equation}
where $D f$ is the G\^ataux derivative of $f$ and $C^1_b(E)$ denotes 
the space of all bounded, G\^ataux-differentiable functions with bounded G\^ataux derivative. 

We now give sufficient conditions on a strong Feller semigroup $\cT$, not necessarily the Ornstein-Uhlenbeck semigroup, 
such that $(\cT, B)$ satisfies Hypothesis \ref{hyp1}.

\begin{hyp}\label{hyp.sf}
Let $\cT$ be a strong Feller semigroup with full generator $\cL$ such that 
\begin{enumerate}
 \item $\cT(t)f \in C^1_b(E)$ for all $f \in B_b(E)$ and $t>0$;
 \item there exists $\varphi \in L^1_\mathrm{loc}([0,\infty))$ such that $\sup_{x \in E}\|D \cT(t)f(x)\|_{E^*} 
\leq \varphi (t) \|f\|_\infty$ for all $f \in B_b(E)$ and $t >0$;
 \item For all $x, y \in E$, the linear functional $f \mapsto \dual{y}{D \cT(t)f(x)}$ is given by 
a measure $\mu = \mu_{t,x,y}\in \M (E)$, i.e.\ $\dual{y}{D \cT(t)f(x)}=\dual{f}{\mu}$ for all $f \in B_b(E)$.
\end{enumerate}
\end{hyp}

\begin{prop}\label{p.verify}
Assume Hypothesis \ref{hyp.sf}. Then $\cT$ together with $B$, defined by \eqref{eq.b} satisfies 
Hypothesis \ref{hyp1}. 
\end{prop}

\begin{proof}
Let us first show that $Bf \in B_b(E)$ whenever $f \in C^1_b(E)$.

Since $\dual{y}{D f(x)} = \lim_{h\to 0}h^{-1}[f(x+hy) - f(x)]$, it follows that 
$x \mapsto D f(x)$ is scalarly $E$-measurable whenever $f \in C^1_b(E)$. Since $F$ is strongly measurable, 
there exists a sequence of simple functions $F_n = \sum_{k=1}^{K_n}\one_{A_{kn}}x_{kn}$ converging pointwise to $F$. 
Hence
\[ Bf(x) = \dual{F(x)}{D f(x)} = \lim_{n\to\infty}\dual{F_n(x)}{D f(x)} = 
 \lim_{n\to \infty} \sum_{k=1}^{K_n}\one_{A_{kn}}(x)\dual{x_{nk}}{D f(x)}
\]
for all $x \in E$, proving that $Bf$ is measurable as the pointwise limit of measurable functions.\smallskip 

Consequently, $B\cT (t)f$ is a well-defined element of $B_b(E)$ for all $t>0$ and $f \in B_b(E)$. By Hypothesis 
\ref{hyp.sf} (3), if $f_n \weak f$, then $\dual{y}{D\cT (t)f_n(x)} \to \dual{y}{D\cT f(x)}$ for all $x,y\in E$. 
Using this with $y = F(x)$, it follows 
that $B\cT (t)f_n\weak B\cT (t)f$. This proves that (2) in Hypothesis \ref{hyp1}
is satisfied.\smallskip 

To prove Hypothesis \ref{hyp1} (3), it suffices to show that
$s \mapsto B\cT(s)f (x)$ is measurable for all $x \in E$. But this follows immediately from the equation
\[ B\cT(s)f(x) = \lim_{h \to 0} \frac{\cT(s)f(x+hF(x)) - \cT(s)f(x)}{h}\] 
since the functions on the right-hand side are measurable in $s$.\smallskip 

Hypothesis \ref{hyp1} (4) follows directly from Hypothesis \ref{hyp.sf} (2) and the boundedness of $F$.\smallskip 

It remains to prove that $D(\cL ) \subset C^1_b(E)$ and that Hypothesis \ref{hyp1} (1) holds. To that end, 
fix $\eps >0$. Then 
\[
  \cR (\lambda )f = e^{-\lambda \eps} \cT (\eps ) \cR (\lambda) f + \int_0^\eps e^{-\lambda t}\cT (t)f\, dt \, .
\]
By the above, $e^{\lambda \eps} B\cT (\eps )\cR (\lambda)$ defines an operator in $\cL (B_b(E), \sigma )$.
Concerning the second term, we first note that
\[ \Big|\frac{\cT(t) f(x+ hy) - \cT f(x)}{h}\Big| = \Big|\frac{1}{h}\int_0^h \dual{y}{D \cT (t)f(x+ sy)}\, ds \Big|
 \leq \varphi (t)\|y\|\, .
\]
Hence, by dominated convergence, it follows that $\int_0^\eps e^{-\lambda t}\cT (t)f\, dt \in C^1_b(E)$ and 
\[ \Big\langle y, D\int_0^\eps e^{-\lambda t}\cT (t)f \, dt (x)\Big\rangle 
 = \int_0^\eps e^{-\lambda t}\dual{y}{D\cT (t)f(x)}\, dt\, .
\]
Consequently, $B\int_0^\eps e^{-\lambda t}\cT (t)f\, dt = \int_0^\eps e^{-\lambda t}B\cT(t)f\, dt$, and from this
representation one infers that also $B\int_0^\eps e^{-\lambda t}\cT (t)f\, dt$ defines an operator in 
$\cL (B_b(E), \sigma )$. The proof is now complete. 
\end{proof}

For later use we note that the proof of Proposition \ref{p.verify} yields for $\eps >0$
\begin{equation}\label{eq.resrep}
  B\cR (\lambda )f = e^{-\lambda \eps} B\cT (\eps ) \cR (\lambda) f + \int_0^\eps e^{-\lambda t}B\cT (t)f\, dt \, .
\end{equation}

Let us now turn to the Ornstein-Uhlenbeck semigroup.
It is well-known that if equation $[A, 0, G]$ is well-posed on $E$, then for every $t>0$ the operator
$Q_t \in \cL (E^*, E)$, defined by 
\[ Q_tx^* := \int_0^t S(s)GG^*S^*(s)x^*\, ds \quad \forall\, x^* \in E^*\]
is the covariance operator of a centered Gaussian measure $\mu_t$ on $(E, \B (E))$, 
see \cite{vNW05}. Furthermore, 
the transition semigroup $\cT_{\OU}$ associated with equation $[A, 0, G]$ is given by
\[ \cT_{\OU}(t)f(x) \int_E f(S(t)x +z)\, d\mu_t(z) = \int_E f(z) d\mathcal{N}_{(S(t)x, Q_t)}(z)\quad f \in B_b(E)\,,\]
where $\mathcal{N}_{(m,Q)}$ denotes the Gaussian measure with mean $m$ and covariance $Q$.

Let us now recall some facts about the reproducing kernel Hilbert space $H_{Q_t}$ associated to $Q_t$.
$H_{Q_t}$ is defined as the completion of the range of $Q_t$ under the inner product $\ip{}{}_{Q_t}$, defined by
\[ \ip{Q_tx^*}{Q_ty^*}_{Q_t} := \dual{Q_tx^*}{y^*}\quad x^*, y^* \in E^*\, .\]
It is well-known that the identity on the range of $Q_t$ extends to a continuous injection from 
$H_{Q_t}$ to $E$, hence $H_{Q_t}$ may be viewed as a subspace of $E$. We also recall that the map 
$\phi^{Q_t} : Q_tx^* \mapsto \dual{\cdot}{x^*}$ extends to an isometry from $H_{Q_t}$ to $L^2(\mu_t)$. 
This extension is called the \emph{Paley-Wiener map}. With slight abuse of notation, we denote the extension also 
by $\phi^{Q_t}$. We will write $\phi_h^{Q_t}$ instead of $\phi^{Q_t}(h)$ for $h \in H_{Q_t}$.

It was seen in \cite[Corollary 2.4]{vN97}, see also \cite[Appendix B]{dpz_stoch} for the Hilbert space case, 
that $\cT_\OU$ is a strong Feller semigroup if and only if 
\[ S(t)E \subset H_{Q_t} \quad \forall\, t>0\, .\]\
This condition also has a control theoretic interpretation, cf.\ \cite{vN05}.
We now collect some properties of the Ornstein-Uhlenbeck semigroup that will be used in what follows.

\begin{prop}\label{p.ou}
Under the assumptions of Theorem \ref{t.main}, 
\begin{enumerate}
 \item $t \mapsto \|S(t)\|_{\cL (E, H_{Q_t})}$ is measurable.
 \item $S(t)$ is a compact operator from $E$ to $E$ for all $t>0$.
 \item $\cT_\OU$ satisfies Hypothesis \ref{hyp.sf};
 \item $\cT_\OU$ is irreducible;
 \item $\cT_\OU|_{C_b(E)}$ has $\beta_0(E)$-continuous orbits;
 \end{enumerate}
\end{prop}

\begin{proof}
(1) By \cite[Theorem 1.4]{vN97}, for $t,s >0$ the operator 
$S(s)$ maps $H_{Q_{t}}$ to $H_{Q_{t+s}}$ and $\|S(s)\|_{\cL (H_{Q_t}, H_{Q_{t+s}})} \leq 1$.
Consequently, 
\[ \|S(t+s)\|_{\cL (E, H_{Q_{t+s}})} \leq \|S(t)\|_{\cL (E, H_{Q_t})} \|S(s)\|_{\cL (H_{Q_t}, H_{Q_{t+s}})}
\leq \|S(t)\|_{\cL (E, H_{Q_t})}\, . \] 
It follows that $t \mapsto \|S(t)\|_{\cL (E, H_{Q_t})}$ is decreasing and hence 
measurable.\medskip 

(2) follows from \cite[Proposition 5.7]{vN97}.\medskip 

(3) It follows from \cite[Theorem 2.3]{vN97} that if $\cT_\OU$ is a strong Feller semigroup, then 
$\cT(t)f \in C^1_b(E)$ for all $f \in B_b(E)$, i.e.\ 
(1) in Hypothesis \ref{hyp.sf} is satisfied. Furthermore, for $f \in B_b(E)$ and  $x, y \in E$ we have
\[ \dual{y}{D\cT_\OU(t)f(x)} = \int_E f(S(t)x + z) \phi_{S(t)y}^{Q_t}(z)\, d\mu_t(z).\]
This yields that (3) in Hypothesis \ref{hyp.sf} is satisfied. From the above
equation we also obtain that
\[ |\dual{y}{D\cT_\OU(t)f(x)}| \leq \|f\|_\infty \|\phi_{S(t)y}^{Q_t}\|_{L^2(\mu_t)} = \|f\|_\infty \|S(t)y\|_{H_{Q_t}}\, .
\]
Hence, putting $\varphi (t) := \|S(t)\|_{\cL (E, H_{Q_t})}$ which belongs 
to $L^1_\mathrm{loc} ([0,\infty ))$ by assumption, 
 $\sup_{x\in E}\|D\cT_\OU (t)f(x)\| \leq \varphi (t) \|f\|_\infty$, proving (2) in Hypothesis \ref{hyp.sf}.\medskip

(4) Since $H_{Q_s} \subset H_{Q_t}$ for $0<s\leq t$, see \cite[Proposition 1.3]{vN97}, it follows that 
$\{ S(s)x\, :\, x \in E\, ,\, 0<s\leq t\} \subset H_{Q_t}$. Since $\bS$ is strongly continuous, $H_{Q_t}$ is 
norm-dense in $E$ for all $t>0$. It is well-known, see \cite[Theorem 3.6.1]{bogachev}, that the support of a Gaussian 
measure on $E$ is the closure of its reproducing kernel Hilbert space in $E$, shifted by the mean of the measure.
Hence, the support of $\mathcal{N}_{(S(t)x, Q_t)}$ is $S(t)x + \overline{H_{Q_t}} = E$. Consequently, every 
open set $U \subset E$ has positive measure with respect to $\mathcal{N}_{(S(t)x, Q_t)}$ 
and thus, for every $x \in E$ and $t>0$ we have $\cT_\OU (t)\one_U(x) = \mathcal{N}({S(t)x, Q_t)}(U) > 0$.\medskip 

(5) is \cite[Theorem 6.2]{vNG03}.
\end{proof}

In the rest of this article, $\cL_\OU$ denotes the full generator of $\cT_\OU$. By what was done so far, 
$\cL_\OU + B$ generates a strong Feller semigroup $\scrP$ and $\scrP|_{C_b(E)}$ has $\beta_0(E)$-continuous orbits.
The semigroup $\scrP$ is our candidate for the transition semigroup associated to equation $[A,F,G]$. 

However, up to now, it is not even clear whether $\scrP$ is Markovian, which of course is necessary to be a 
transition semigroup. 

\section{Well-posedness}\label{sect.well}

In order to finish the proof of Theorem \ref{t.main}, we have to link the semigroup $\scrP$, generated by $\cL_\OU +B$, 
to the solutions of equation $[A,F,G]$. This will be done by considering the associated (local) 
martingale problems.

\subsection{Martingale problems}
If $(M, d)$ is a complete separable metric space, then also $C([0, \infty ); M)$ is a complete separable 
metric space for the metric 
\[ \rho (\bx , \by ) := \sum_{k=1}^\infty 2^{-k} \sup_{t \in [0,k]} \big[1\wedge d (\bx (t), \by (t))\big] \]
The Borel $\sigma$-algebra of $(C([0,\infty ); M), \rho )$ will be denoted by 
$\cB$. It is well-known, see \cite[Lemma 16.1]{kallenberg}, that $\cB = \sigma (\bx (s) : s\geq 0 )$. Here, in 
slight abuse of notation, we have identified $\bx (s)$ with the $M$-valued random variable $\bx \mapsto \bx (s)$. We 
shall do so in what follows without further notice.
The filtration generated by these coordinate mappings is denoted by $\mathbb{B} := (\cB_t)_{t\geq 0}$, i.e.\ 
$\cB_t := \sigma (\bx (s)\, : 0\leq s \leq t )$.\medskip

We now recall the following definition, cf.\ \cite[Section 4.3]{ek}.

\begin{defn}\label{d.mart}
Let $\cL \subset B_b(M)\times B_b(M)$.
A measure $\PP \in \cP ( C([0,\infty ); M))$ is said to 
\emph{solve the martingale problem for $\cL$}, if for all $(f,g) \in \cL$, the process $\bM^{f,g}$, defined
by
\[ \big[\bM^{f,g}(\bx)\big](t) := f(\bx (t)) - \int_0^t g (\bx (s))\, ds \]
is a $\mathbb{B}$-martingale under $\PP$. Given a measure $\mu \in \cP (M)$, we say that 
$\PP$ solves the martingale problem for $(\cL, \mu )$ if $\PP$ solves the martingale problem for 
$\cL$ and $\PP (\bx (0) \in \cdot ) = \mu (\cdot )$.

We say that \emph{uniqueness in law} holds for the martingale problem for $\cL$, if for all $\mu \in \cP (M)$ whenever
$\PP_1$ and $\PP_2$ are solutions to the martingale problem for $(\cL, \mu )$, then
$\PP_1 = \PP_2$.

We say that the martingale problem for $\cL$ is \emph{well-posed} if for every $x \in M$ there exists 
a unique solution $\PP_x \in \cP (C([0,\infty ); M)$ of the martingale problem for $(\cL, \delta_x)$. 
\end{defn}

We will also have occasion to consider sets $\hat{\cL}$ which do not necessarily consist of bounded 
functions. Given a subset $\hat{\cL} \subset C(M)\times B(M)$, we say $\hat{\cL}$ is \emph{admissible} if 
for all $(f, g) \in \hat{\cL}$ the function $g$ is bounded on compact subsets of $M$. Note that if 
$\hat{\cL}$ is admissible, then $\bM^{f,g}$ is well-defined for all $(f,g) \in \hat{\cL}$.

A measure $\PP$ 
\emph{solves the local martingale problem for $\hat{\cL}$} if $\bM^{f,g}$ is a local $\mathbb{B}$-martingale under 
$\PP$ for all $(f,g) \in \hat{\cL}$. 
The terms `uniqueness in law' and `well-posedness' will also be used for local martingale problems with 
the obvious meaning.
\medskip 

The basic link between strong Feller semigroups and martingale problems is the following result.

\begin{thm}\label{t.unique}
Suppose that $\cL$ is the full generator of the strong Feller semigroup $\cT$. If $\PP$ solves the 
martingale problem for $\cL$, then for $f \in B_b(M)$ we have
\begin{equation}\label{eq.dist}
  \expect_\PP \big[ f(\bx (t+s))\, \big|\cB_s\big] = \cT (t) f(\bx (s)) \quad \PP-a.s.
\end{equation}
\end{thm}

\begin{proof}
This result was proved in \cite[Theorem 4.4.1]{ek} under the more restrictive assumption 
that there exists a closed subspace $\mathbf{L}$ of $B_b(E)$ which is in some sense `large enough'
such that $\mathbf{L}$ is invariant under $\cT$ and $\cT|_{\mathbf{L}}$ is a strongly continuous contraction semigroup.

We would like to use this result with $\mathbf{L} = C_b(M)$, but, in general, $\cT|_{C_b(M)}$ is \emph{not}
strongly continuous. However, inspecting the proof in \cite{ek}, we see that 
all that was used about $\cT|_{\mathbf{L}}$ in the proof was that the domain of the 
generator is (sequentially) $\|\cdot\|_\infty$-dense in $\mathbf{L}$ and that the Post-Widder inversion formula
\[ \cT(t)f = \|\cdot\|_\infty-\lim_{n\to\infty} \Big[ \frac{n}{t}R\Big(\frac{n}{t}, \cL|_{\mathbf{L}}\Big)\Big]^nf \]
holds for $f \in \mathbf{L}$. 

Since $\cT|_{C_b(M)}$ is $\sigma$-continuous at 0, it follows, cf.\ \cite{k09}, that both properties 
are true with the $\|\cdot\|_\infty$-topology replaced with the $\sigma$-topology. The reader may check that 
the proof of \cite[Theorem 4.4.1]{ek} still works under these weaker assumptions.
\end{proof}

As a consequence of Theorem \ref{t.unique}, we obtain uniqueness in law for the martingale problem 
for $\cL$, cf.\ \cite[Theorem 4.4.2]{ek}. We note that if the martingale problem for $\cL$ is 
well-posed, then $\cT$ is Markovian, which follows easily from equation \eqref{eq.dist}.\medskip 

By Theorem \ref{t.unique}, to prove uniqueness in law for equation $[A, F, G]$, it suffices to 
prove that the law of any solution of $[A, F, G]$ solves the martingale problem for $\cL_\OU + B$. 
In fact, the following lemma shows that it suffices to consider a core for $\cL_\OU + B$.
 
\begin{lem}\label{l.mart}
Let $\cL \subset B_b(M)\times B_b(M)$. If $\PP \in \mathcal{P}(C([0,\infty); M))$ solves the martingale problem
for $\cL$, then $\PP$ solves the martingale problem for $\overline{\cL}^{\sigma\times\sigma}$. 
\end{lem}

\begin{proof}
$\PP$ solves the martingale problem for $\cL$ if and only if 
\[ \int \one_A(\bx) \Big( f(\bx (t)) - f(\bx (s)) - \int_s^t g(\bx (r))\, dr \,\Big) d\PP(\bx ) = 0 \]
for all $0\leq s < t, A \in \cB_s$ and $(f,g) \in \cL$. Now fix $s,t$ and $A$.

Define the measure $\PP^A$ by $\PP^A(B) := \PP(B\cap A)$ 
and let $\mu_t$ be the distribution of $\bx (t)$ under $\PP^A$. Then the above equation is equivalent to 
\[ \dual{f}{\mu_t} - \dual{f}{\mu_s} - \int_s^t \dual{g}{\mu_r}\, dr = 0 \]
for all $(f,g) \in \cL$. We note that the function $r \mapsto \mu_r$ is $B_b(E)$-integrable on $(s,t)$. 
To see this, note that for $h \in C_b(E)$ the function $r \mapsto \dual{h}{\mu_r}$ is continuous, 
which is easy to see by dominated convergence using that $\bx$ is continuous. By a monotone class argument, 
it follows that $r \mapsto \dual{h}{\mu_r}$ is measurable for all $h \in B_b(E)$. This shows that $r \mapsto 
\mu_r$ is scalarly $B_b(E)$-measurable. Since $\|\mu_r\|_{\TV} \leq 1$ for all $r$, it follows 
that $r\mapsto \mu_r$ is $B_b(E)$-integrable on $(s,t)$, i.e.\ there exists 
a measure $I_s^t(\mu_\cdot) \in \M (E)$ such that 
$\int_s^t \dual{h}{\mu_r}\, dr =  \dual{h}{I_s^t(\mu_\cdot)}$ for all $h \in B_b(E)$.
Consequently, the above equation is equivalent to
\[ \dual{f}{\mu_t} - \dual{f}{\mu_s} - \dual{g}{I_s^t(\mu_\cdot)} = 0 \]
for all $(f,g) \in \cL$. Clearly, this equation is also satisfied for $(f,g) \in \overline{\cL}^{\sigma\times \sigma}$.
Since $0\leq s < t$ and $A \in \cB_s$ were arbitrary, it follow that $\PP$ solves the martingale problem 
for $\overline{\cL}^{\sigma\times\sigma}$.
\end{proof}

\subsection{Uniqueness in law} 
We now return to our study of equation $[A,F,G]$.

A \emph{mild martingale solution} of equation $[A, F, G]$ is a tuple $((\Omega, \Sigma, \P), \FF, \bX, W_H)$, 
where $(\Omega, \Sigma, \P)$ is a probability space, endowed with the filtration $\FF$, $W_H$ is an 
$H$-cylindrical Wiener process with respect to $\FF$ and $\bX$ is a continuous, progressive, $E$-valued 
process such that $\P$-a.s.
\[ X(t) = X(0) + \int_0^tS(t-s)F(X(s))\, ds + \int_0^tS(t-s)GdW_H(s)\]
for all $t \geq 0$. For the definition of the stochastic integral and a characterization of stochastic integrability 
we refer to \cite{vNW05}.

Let us recall the definition of the associated local martingale problem from \cite{k10b}: 

We denote by $\mathscr{D} (A, C^2)$ the vector space of all functions $f : E \to \CR$ of the form
\[ f(x) = \psi (\dual{x}{x_1^*}, \ldots, \dual{x}{x_n^*}) \]
where $n \in \CN, \psi \in C^2(\CR^n)$ and $x_1^*, \ldots, x_n^* \in D(A^*)$. Note that $A$ denotes the 
generator of the semigroup on $\tilde{E}$ so that $A^*$ is the weak$^*$-generator of the adjoint semigroup 
on $\tilde{E}^*$. In particular, $D(A^*) \subset \tilde{E}^*$.

For $f = \psi (\dual{\cdot}{x_1^*}, \ldots, \dual{\cdot}{x_n^*}) \in \mathscr{D}(A, C^2)$ we put
\[
 \begin{aligned}
  L_{[A, F, G]} f(x) := & \sum_{k=1}^n \frac{\partial\psi}{\partial u_k}(\dual{x}{x_1^*}, \ldots , 
\dual{x}{x_n^*})\cdot \big[ \dual{x}{A^*x_k^*} + \dual{F(x)}{x_k^*}\big]\\
& \quad + \half \sum_{k,l=1}^n [G^*x_k^*, G^*x_l^*]_H
\frac{\partial^2\psi}{\partial u_k\partial u_l} (
\dual{x}{x_1^*}, \ldots, \dual{x}{x_n^*})\, .
 \end{aligned}
\]
We define $\hat{\cL}_{[A, F, G]} := \{ (f, L_{[A,F,G]}f) \, : \, f \in \mathscr{D}(A, C^2)\}$ and denote the part of 
$\hat{\cL}_{[A,F,G]}$ in $B_b(E)$ by $\cL_{[A,F, G]}$. 

Under the assumption that $S(t)$ maps $\tilde{E}$ to $E$ such that
$\int_0^T\|S(t)\|_{\cL (\tilde{E}, E)}^2\, dt < \infty$ for all $T>0$, i.e.\ \eqref{eq.hyp2} holds, it 
was proved in \cite{k10b}, that there is a one-to-one correspondence between mild martingale solutions 
of $[A,F,G]$ and solutions of the local martingale problem for $\hat{\cL}_{[A,F,G]}$. More precisely, if 
$\bX$ is a mild martingale solution of $[A,F,G]$, then its distribution $\PP_\bX$ solves the local 
martingale problem for $\hat{\cL}_{[A,F,G]}$. Conversely, if $\PP$ solves the local martingale problem 
for $\hat{\cL}_{[A,F,G]}$, then there exists a mild martingale solution of $[A,F,G]$ with distribution 
$\PP$.

In view of this equivalence, we will say that equation $[A,F,G]$ is \emph{well-posed} resp.\ \emph{uniqueness 
in law holds} for 
equation $[A,F,G]$ if the local martingale problem for $\hat{\cL}_{[A,F,G]}$ is well-posed 
resp.\ uniqueness in law holds for the local martingale problem for $\hat{\cL}_{[A,F,G]}$.

By \cite[Theorem 3.8]{k10b}, if $[A, F, G]$ is well-posed, then all mild martingale solutions of 
equation $[A, F, G]$ are Markov processes with a common transition semigroup $\cT$. 
Moreover, for every $\mu \in \cP (E)$, there exists a mild martingale solution of $[A,F,G]$ with initial 
distribution $\mu$; the distribution of this solution is uniquely determined by $\mu$ and the 
coefficients $A, F$ and $G$.

\begin{rem}
It should be noted that if equation $[A,0,G]$ is well-posed, then there exists a corresponding Ornstein-Uhlenbeck process 
with continuous paths in $E$. Not every Ornstein-Uhlenbeck processes has a modification with continuous paths, 
see \cite{immtz90} for an example with unbounded $G$, i.e.\ $\tilde{E}\neq E$. 
In the case where $\tilde{E} = E$, it seems to be an open question whether an $E$-valued Ornstein-Uhlenbeck process 
has a continuous modification. However, to the best of our knowledge, there is no example known of an Ornstein-Uhlenbeck process 
which does not have a continuous modification but has a transition semigroup which is strongly Feller.  
\end{rem}

We now prove that for bounded $F$ the only possible transition semigroup for 
equation $[A,F,G]$ is the one generated by $\cL_\OU + B$.

\begin{thm}\label{t.uil}
Under the assumptions of Theorem \ref{t.main}, uniqueness in law holds for equation $[A,F,G]$. 
Furthermore, in case of well-posedness, its transition semigroup is generated by $\cL_\OU + B$.
\end{thm}

\begin{proof}
In view of Theorem \ref{t.unique}, it suffices to prove that a solution of the local martingale problem 
for $\hat{\cL}_{[A,F,G]}$ solves the martingale problem for $\cL_\OU + B$.

Let $\PP$ be a solution of the local martingale problem for $\hat{\cL}_{[A,F,G]}$ and $(f,g) \in \cL_{[A,F,G]}$. 
Since $(f,g) \in \hat{\cL}_{[A,F,G]}$, the process $\bM^{f,g}$ is a local $\mathbb{B}$-martingale under $\PP$. 
By the boundedness of $f$ and $g$, an approximation argument shows that under $\PP$, the process $\bM^{f,g}$ is actually a 
$\mathbb{B}$-martingale. Hence, $\PP$ solves the martingale problem for $\cL_{[A,F,G]}$. An easy computation shows 
that $\cL_{[A,F,G]} = \cL_{[A, 0, G]} + B$, where $B$ is defined by \eqref{eq.b}. 
Standard arguments, see \cite[Theorem 6.6]{vNG03}, \cite[Lemma 2.6]{gk}, show 
that the domain of $\cL_{[A, 0, G]}$ is invariant under $\cT_\OU$ and $\beta_0(E)$-dense 
in $C_b(E)$, hence $\sigma$-dense in $B_b(E)$. Thus, by Proposition \ref{p.core}, 
$\cL_{[A, 0, G]}$ is a core for $\cL_\OU$.\smallskip

We claim that $\cL_{[A, 0, G]} + B$ is a core for $\cL_\OU + B$. 

To see this, let $(f, g) \in \cL_\OU + B$. Then $(f, g- Bf) \in \cL_\OU $. Since $\cL_{[A, 0, G]}$ is a core 
for $\cL_\OU$, there is a net $(u_\alpha, v_\alpha )
\in \cL_{[A, 0, G]}$ such that $u_\alpha \weak f$ and $v_\alpha\weak\tilde{g} := g- Bf$.  

Note that for every  $\alpha$ we have $(u_\alpha, v_\alpha + B u_\alpha ) \in \cL_{[A, 0, G]} + B$. 
Furthermore, 
\[ Bu_\alpha = B\cR_\OU (\lambda )(\lambda u_\alpha - v_\alpha) \weak B\cR_\OU (\lambda )(\lambda f - \tilde{g}) = Bf\, , \]
since $B\cR_\OU(\lambda )$ is $\sigma$-continuous. Hence, 
$v_\alpha + Bu_\alpha\weak g-Bf + Bf = g$. Since $u_\alpha$ converges to $f$ with respect to $\sigma$, 
we have proved that $(f,g)$ belongs to the $\sigma\times\sigma$-closure of $\cL_{[A, 0, G]}+ B$.\smallskip 

By Lemma \ref{l.mart}, $\PP$ is a solution to the martingale problem for $\cL_\OU + B$. 
\end{proof}

\subsection{Existence of solutions}
We are nearly ready to finish the proof of Theorem \ref{t.main}. It remains to prove the existence of 
solutions to equation $[A,F,G]$. The following Theorem is the crucial ingredient to finish the proof.

\begin{thm}\label{t.limit}
Under the assumptions of Theorem \ref{t.main}, let $F_n$ be a bounded sequence that converges pointwise to $F$ 
such that $[A, F_n, G]$ is well-posed for all $n \in \CN$. Then $[A, F, G]$ is well-posed.
\end{thm}

\begin{proof}
We denote by $\scrP_n$ the transition semigroup of $[A, F_n, G]$ and by $\scrP$ the semigroup generated by $\cL_\OU + B$. 
Note that by Theorem \ref{t.uil}, the full generator of $\scrP_n$ is $\cL_\OU + B_n$, where 
$B_n$ is defined by \eqref{eq.b} with $F$ replaced with $F_n$.

We proceed in several steps.\medskip 

{\it Step 1 --}
We prove that for $\lambda >0$ large enough $(\lambda - \cL_\OU - B_n)^{-1}f$ converges to $(\lambda - \cL_\OU - B)^{-1}f$ 
with respect to $\beta_0(E)$
for all $f \in B_b(E)$.

We again put $\varphi (t) := \|S(t)\|_{\cL (E, H_{Q_t})}$. Let us first note that
\[
\begin{aligned}
 |B_n\cT_\OU (t)f(x) &- B\cT_\OU(t)f(x)| =  \Big| \int_E f(S(t)x+z)\phi^{Q_t}_{S(t)F_n(x) - S(t)F(x)}\, d\mu_t(z)\Big|\\
& \leq  \|f(S(t)x + \cdot \,)\|_{L^2(\mu_t)} \|\phi_{S(t)(F_n(x) - F(x))}^{Q_t}\|_{L^2(\mu_t)}\\
& \leq  \varphi (t)\|f(S(t)x + \cdot \,)\|_{L^2(\mu_t)}\|F_n(x) - F(x)\|_E\, ,
\end{aligned}
\]
where we have used that $\phi^{Q_t}$ is an isometry from $H_{Q_t}$ to $L^2(\mu_t)$ in the last step.
Hence, by \eqref{eq.resrep}, 
\[  |\big[B_n\cR_\OU (\lambda ) - B\cR_\OU(\lambda )\big](x)| 
\leq \|f\|_\infty \Big[
 e^{\lambda \eps} \varphi (\eps) + \int_0^\eps e^{-\lambda t} \varphi (t)\, dt\Big]\cdot \|F_n(x) - F(x)\|
\to 0
\]
as $n \to \infty$, by dominated convergence. 

Next note that 
\[ 
\begin{aligned}
|(B_n\cR_\OU(\lambda ))^2f(x) - & (B\cR_\OU(\lambda ))^2f(x)| \leq 
 \big|B_n\cR_\OU(\lambda )\big[ B_n\cR_\OU(\lambda)f - B\cR_\OU(\lambda )f\big](x)\big|\\
 &+ \big|\big[B_n \cR_\OU(\lambda) - B \cR_\OU(\lambda )\big]
B\cR_\OU(\lambda )f(x)\big|
\end{aligned}
\]
where the last term converges to 0 as $n \to \infty$ by the above. Let us put $g_n := B_n\cR_\OU(\lambda )f$ 
and $g := B\cR_\OU(\lambda )f$, so that $g_n \weak g$. Let $C := \sup_n\|F_n\|_\infty$. 
Arguing similar as above, we see that
\[
\begin{aligned}
|B_n\cR_\OU(\lambda )(g_n-g)(x)| & \leq e^{-\lambda \eps}C\|g_n(S(\eps)x +\cdot ) - g(S(\eps )x + \cdot )\|_{L^2(\mu_\eps)}\\
& + C \int_0^\eps e^{-\lambda t}\varphi (t) \|g_n(S(t)x +\cdot ) - g(S(t)x+ \cdot )\|_{L^2(\mu_t)}\, dt
\end{aligned}
\]
which converges to $0$ as $n \to \infty$ by dominated convergence.

Iterating, it follows that $[B_n\cR_\OU(\lambda )]^kf \weak [B\cR_\OU(\lambda )]^kf$ for all $k \in \CN$. 
Now we pick $\lambda$ large enough 
so that $\|B_n\cR_\OU(\lambda )\| \leq c<1$ for all $n \in \CN$. Since $(I - B_n\cR_\OU(\lambda ))^{-1} 
= \sum_{k=0}^\infty [B_n\cR_\OU(\lambda )]^k$, where the series converges in operator norm, 
we see that 
$(I-B_n\cR_\OU(\lambda ))^{-1}f \weak (I - B\cR_\OU(\lambda ))^{-1}f$ for all $f \in B_b(E)$. 

It follows that 
\[ (\lambda - \cL_\OU - B_n)^{-1}f = \cR_\OU(\lambda ) (I - B_n\cR_\OU(\lambda ))^{-1}f \weak 
(\lambda - \cL_\OU - B)^{-1}f .\]
In fact, since $\cR_\OU(\lambda )$ is ultra-Feller, cf.\ \cite[Proposition I.5.12]{revuz}, it follows that 
the convergence is with respect to $\beta_0(E)$.
To simplify notation, we write $\cR_n(\lambda ) := (\lambda - \cL_\OU - B_n)^{-1}$ and $\cR_\infty (\lambda ) 
:= (\lambda - \cL_\OU - B)^{-1}$ for the rest of this proof.
\medskip 

{\it Step 2 --} We construct solutions of the martingale problem for $\cL_\OU + B$.

Fix $x \in E$ and let $\PP_n$ be a solution of the local martingale problem for $\hat{\cL}_{[A, F_n, G]}$ with 
initial distribution $\delta_x$. 

By \cite{k10b}, there exists a mild martingale solution $((\Omega_n, \Sigma_n, \P_n), \FF_n, \bX_n, W_H^n)$
 of equation $[A, F_n, G]$ with 
distribution $\PP_n$. Note that we have
\[ X_n(t) = S(t)x + \int_0^t S(t-s)F_n(X_n(s))\, ds + \int_0^t S(t-s)GdW_H^n(s) \]
$\P_n$-almost surely for all $t\geq 0$. Since $\bS$ is immediately compact, the map 
$\phi \mapsto [t \mapsto \int_0^t S(t-s)\phi (s)\, ds]$ is compact as an operator from $L^\infty((0,T); E)$ 
to $C([0,T]; E)$, see \cite[Proposition 1]{gg94a}. By the uniform boundedness of $F_n$, the paths of 
$F_n(X_n(\cdot))$ belong to a bounded subset of $L^\infty (0,T; E)$ almost surely for all $n$. Hence, there is a compact 
subset $\mathscr{C}_T$ of $C([0,T]; E)$ such that for all $n \in \CN$ we have 
$\bY_n := \int_0^\cdot S(\cdot -s)F_n(X_n(s))\, ds \in \mathscr{C}_T$, $\P_n$-almost surely.

Since equation $[A, 0, G]$ is well-posed, the distribution $\mathbf{Q}$ of  $\bX_n - \bY_n=: \mathbf{Z}_n$ is independent of 
$n$, it is the distribution of the solution of 
$[A, 0, G]$ starting at $x$. Hence, given $\eps >0$, there exists a compact 
set $\mathscr{K}_{T,\eps} \subset C([0,T]; E)$ such that 
$\P_n(\mathbf{Z}_n^{[0,T]} \in \mathscr{K}_{T,\eps}) = \mathbf{Q}(\bx|_{[0,T]} \in \mathscr{K}_{T,\eps}) \geq 1-\eps$
for all $n\in \CN$.

Note that the set $\mathscr{C}_T + \mathscr{K}_{T,\eps}$ is compact. Furthermore, 
\[ \PP_n^{[0,T]}((\mathscr{C}_T + \mathscr{K}_{T,\eps})^c) = \P_n\big( X(t) \not\in \mathscr{C}_T + \mathscr{K}_{T,\eps} 
 \quad \mbox{for some} \, t \in [0,T]\big) \leq 0 + \eps\, .
\]
This proves that the restrictions $\PP_n^{[0,T]}$ of $\PP_n$ to $C([0,T]; E)$ are tight. Since $T$ was arbitrary, 
the measures $\PP_n$ are tight, see \cite[Proposition 16.6]{kallenberg}.

Passing to a subsequence, we may assume that $\PP_n$ converges weakly to a measure $\PP$. We claim that 
$\PP$ solves the martingale problem for $\cL_\OU + B$.\medskip  

To prove this, we proceed as in \cite[Lemma 3.9]{k10b}. 
Fix $0\leq r_1<r_2< \cdots < r_k \leq s <t$, functions $h_1, \ldots , h_k  \in C_b(E)$ and $f \in B_b(E)$ and 
define for $n \in \CN\cup\{\infty\}$ the function $\Phi_n \in C_b(C([0,\infty ); E))$ by 
\[
 \Phi_n (\bx ) := \Big[ \cR_n(\lambda )f (\bx (t)) - \cR_n(\lambda )f(\bx (s))
- \int_s^t [\lambda \cR_n(\lambda )f -f] (\bx (r))\,dr \Big] \prod_{j=1}^kh_j(\bx (r_j))\, ,
\]
where $\lambda >0$ is chosen such that $\cR_n(\lambda )f \to \cR_\infty (\lambda )f$ with respect to $\beta_0(E)$
for all $f \in B_b(E)$. This is possible by Step 1.

We claim that $\Phi_n \to \Phi_\infty$ with respect to $\beta_0(C([0,\infty ); E))$.
To see this, first note that by the uniform boundedness of the $F_n$, there exist constants $C_1$ and $C_2$ such that 
$\|\cR_n(\lambda )\| \leq C_2$ and $\|\lambda \cR_n(\lambda )\| \leq C_2$. Putting $C_3 := 
\prod_{j=1}^k \|h_j\|_\infty$, it follows that $\Phi_n (\bx ) \leq (2C_1 + (t-s)C_2 )C_3$ for all 
$\bx \in C([0,\infty ); E)$, i.e.\ $\Phi_n$ is uniformly bounded. Hence, to prove the 
$\beta_0(C([0,\infty ); E))$-convergence, it suffices to prove uniform convergence on the compact subsets 
of $C([0,\infty ); E)$.

Let $\mathscr{C}$ be a compact subset of $C([0,\infty ); E)$. By the Arzel\`a-Ascoli theorem, there exists 
a compact subset $K$ of $E$ such that $\bx (r) \in K$ for all $0\leq r \leq t$, whenever $\bx \in \mathscr{C}$.
By Step 1, given $\eps>0$ we may choose $n_0$ such that $|\cR_n(\lambda )f(x) - \cR_\infty (\lambda )f(x)| \leq \eps$ for 
all $x \in K$ and $n\geq n_0$. Hence, for $n \geq n_0$ and $\bx \in \mathscr{C}$, 
\[ |\Phi_n(\bx ) - \Phi_\infty (x)| \leq C_3(2\eps + (t-s)\lambda \eps)\, .\]
This proves that $\Phi_n \to \Phi$, uniformly on $\mathscr{C}$. 

Since $\PP_n$ converges weakly to $\PP$, the sequence $\PP_n$ is tight, whence $p(\Psi ) := \sup_n 
\int |\Psi | d\PP_n$ is a $\beta_0(C([0,\infty ); E))$-continuous seminorm. It follows that
\[ |\dual{\Phi_\infty}{\PP} - \dual{\Phi_n}{\PP_n}| \leq 
 |\dual{\Phi_\infty}{\PP -\PP_n}| + p(\Phi_n - \Phi_\infty) \to 0
\]
as $n\to \infty$. Thus $\dual{\Phi_\infty}{\PP} = \lim_{n\to\infty}\dual{\Phi_n}{\PP_n} = 0$, 
since $\PP_n$ solves the martingale problem for $\cL_\OU + B_n$. 

Since the functions $h_j$ and the points $r_j$, as well as the number $k$ were arbitrary, it follows that 
$\bM^{\cR_\infty(\lambda )f, \lambda\cR_\infty (\lambda )f -f}$ is a martingale under $\PP$. 
Noting that
\[ \{ (\cR_\infty(\lambda )f, \lambda \cR_\infty (\lambda )f -f)\, : \, f \in B_b(E)\} = \cL_\OU + B\, ,\] it 
follows that $\PP$ solves the martingale problem for $\cL_\OU + B$.\medskip 

{\it Step 3 --} $\PP$ solves the local martingale problem for $\hat{\cL}_{[A, F, G]}$.

Note that this step finishes the proof since $x \in E$ was arbitrary. By \cite[Lemma 3.7]{k10b}, it 
suffices to prove that $\bM^{f, L_{[A, F, G]}f}$ is a local martingale under $\PP$ for all 
$f$ of the form $f = \psi (\dual{\cdot}{x^*})$, where $\psi \in C_c^\infty (\CR)$ and $x^* \in D(A^*)$.

Let such $\psi$ and $x^*$ be given and put, similarly as in \cite[Proof of Theorem 4.5]{gk}, 
\[ f_n(x) := \frac{\psi (\dual{x}{nR(n,A^*)x^*})}{1 + \dual{x}{R(n, A^*)x^*}^2} \quad n \in \CN\, .\]
  Elementary computations and estimates show that
\begin{enumerate}
 \item $f_n \in D(\cL_\OU + B)$ and $f_n \to f$ pointwise.
 \item $g_n:= L_{[A,F, G]}f_n$ is uniformly bounded on bounded sets and converges pointwise to $g:= L_{[A,F,G]}f$.
\end{enumerate}
Since $\PP$ solves the martingale problem for $\cL_\OU + B$, it follows that for every $n \in \CN$ the process 
$\bM^{f_n, g_n}$ is a martingale under $\PP$. Fix $N \in \CN$ and let 
$\tau_N(\bx ) := \inf\{ t>0 \, : \, \|\bx (t)\|\geq N\}$. By optional sampling, also the stopped 
process $\bM^{f_n, g_n}_{\tau_N}$ is a martingale under $\PP$.

Similar as in Step 2, fix $0 \leq r_1 < r_2<\cdots < r_k < s, t$ and functions $h_1, \ldots , h_k \in C_b(E)$.
We define
\[ \Phi_n(\bx) := \Big[ f_n(\bx (t\wedge \tau_N) - f_n(\bx (s\wedge \tau_N)) 
 - \int_s^t \one_{[0,\tau_N]}(r) g_n(\bx (r))\, dr\Big] \prod_{j=1}^k h_j(r_j)\, ,
\]
and $\Phi$ similarly, replacing $f_n$ with $f$ and $g_n$ with $g$. By the above, $\Phi_n$ is uniformly bounded 
and converges pointwise to $\Phi$. Hence, by dominated convergence, 
\[ \int\Phi \, d\PP = \lim_{n\to \infty}\int \Phi_n\, d\PP = 0 \, .\]
It follows that $\bM^{f, g}_{\tau_N}$ is a martingale under $\PP$. Since $\tau_N \uparrow \infty$ 
almost surely, $\bM^{f, g}$ is a local martingale under $\PP$. This 
finishes the proof.
\end{proof}

We are now ready to finish the proof of Theorem \ref{t.main}:

\begin{proof}[Proof of Theorem \ref{t.main}]
As a consequence of Theorem \ref{t.uil}, uniqueness in law holds for equation $[A,F,G]$, for every 
bounded, measurable $F$. Moreover, if $[A,F,G]$ is well-posed, then the associated transition semigroup 
is generated by $\cL_\OU +B$. This semigroup is strongly Feller by Theorem \ref{t.pert}. That it is irreducible 
follows from the irreducibility of $\cT_\OU$ as in \cite[Proposition 4]{cmg95}. It remains to prove existence of mild 
martingale solutions for equation $[A, F, G]$. 

First consider the situation where $F$ is a bounded Lipschitz continuous function. We fix a probability space $(\Omega, \Sigma, \P)$ 
on which for every $x \in E$ a solution $\bX_x$ of equation $[A,0,G]$ with initial datum $x$ exists.
For $p \in (1,\infty)$ and $T>0$, we define $\Lambda_x : L^p(\Omega; C([0,T]; E)) \to L^p(\Omega; C([0,T]; E))$ by
\[
 \big[\Lambda_x (\phi)\big](t,\omega) := X_x(t,\omega) + \int_0^t S(t-s)F(\phi (s,\omega))\, ds\, .
\]
It is easy to see that $\Lambda_x$ is Lipschitz continuous on $L^p(\Omega; C([0,T];E))$ with Lipschitz constant 
$L(T)$ independent of $x$ and $L(T) \to 0$ as $T\to 0$. Standard arguments show that $\Lambda_x$ has a unique fixed-point 
in $L^p(\Omega; C([0,T];E))$ for all $T>0$. Clearly, this solution solves equation $[A,F,G]$ on the time-interval $[0,T]$.
By uniqueness in law, these solutions can be glued together to a solution defined for all positive times.
\smallskip

Now let $\mathscr{H}$ be the bp-closure, cf.\ \cite[Section 3.4]{ek}, of $\mathsf{Lip}_b(E, E)$, i.e.\ 
the smallest subset of $B_b(E, E)$ which contains $\mathsf{Lip}_b(E, E)$ and if $F_n$ is a bounded sequence in 
$\mathscr{H}$ which converges pointwise to $F \in B_b(E, E)$, then also $F \in \mathscr{H}$.

By Theorem \ref{t.limit}, the set of $F$ such that $[A,F,G]$ is well-posed is bp-closed and, by the above, 
contains $\mathsf{Lip}_b(E, E)$. Hence, $[A,F,G]$ is well-posed for every $F \in \mathscr{H}$. 

For $\psi \in \mathsf{Lip}_b(E, \CR)$ and $x \in E$, we have $F=\psi\otimes x \in \mathsf{Lip}_b(E, E)\subset \mathscr{H}$.
By \cite[Proposition 3.4.2]{ek}, $\psi\otimes x \in \mathscr{H}$ for every $\psi \in B_b(E)$. Since $\mathscr{H}$
is a subspace, cf.\ \cite[Lemma 4.1]{ek}, any simple function belongs to $\mathscr{H}$. But then also any bounded, 
measurable function $F$ belongs to $\mathscr{H}$. Hence, $[A,F,G]$ is well-posed for every bounded, measurable 
function $F$.
\end{proof}

\section{Asymptotic behavior}\label{sect.inv}
Given a Markovian transition semigroup $\cT$ on $B_b(E)$, a probability measure $\mu$ is called an
\emph{invariant measure for $\cT$} if
\[ \int_E \cT(t)f\, d\mu = \int_E f \, d\mu \quad\quad\forall\, t\geq 0\, ,\, f \in B_b(E)\, .\]
By Doob's theorem \cite[Theorem 4.2.1]{dpz_erg}, if $\cT$ is strongly Feller and irreducible, then there 
is at most one invariant measure for $\cT$. Moreover, if there is an invariant measure $\mu$ for $\cT$, 
then it is strongly mixing and $\cT(t)f$ converges pointwise to $\int_E f\, d\mu$ for all $f \in B_b(E)$.

It was proved in \cite[Proposition 4.5]{vNW06}, that if $[A,0,G]$ is well-posed and $\bS$ is uniformly 
exponentially stable, then there exists an invariant measure for the associated transition semigroup.
Conversely, under the Hypothesis of Theorem \ref{t.main}, if there exists an invariant measure for the transition 
semigroup associated to $[A,0, G]$, then $\bS$ is uniformly 
exponentially stable, see \cite[Theorem 5.4]{vN99}. Thus, in the situation of Theorem \ref{t.main}
$\cT_\OU$ has a (necessarily unique) invariant measure if and only if $\bS$ is uniformly exponentially stable.

\begin{thm}\label{t.asymp}
Assume the Hypothesis of Theorem \ref{t.main}. 
If there exists an invariant measure $\mu_\infty$ for $\cT_\OU$, then there exists a unique invariant measure $\mu$ for 
$\scrP$. 
\end{thm}

\begin{proof}
Let $x \in E$ and  $((\Omega, \Sigma , \P), \FF, \bX, W_H)$  be a mild solution of equation $[A,F,G]$
with initial datum $x$, i.e.\ $\P$-almost surely
\[ X(t) = S(t)x + \int_0^tS(t-s)F(X(s))\, ds + \int_0^tS(t-s)GdW_H(t) \]
for all $t \geq 0$. Put $Z(t) := S(t)x + \int_0^tS(t-s)GdW_H(t)$ and $Y(t) := X(t) - Z(t)$.
Then $\mathbf{Z}$ is an Ornstein-Uhlenbeck process with transition semigroup $\cT_\OU$. Since $\cT_\OU$ is 
strongly Feller and irreducible the laws $\mu_{Z(t)}$ of $Z(t)$ converge weakly to the invariant measure $\mu_\infty$
as $t \to \infty$. It follows that the set $\{ \mu_{Z(t)}\, : \, t \geq 1\}$ is tight. 
Consequently, given $\eps >0$, there exists 
a compact set $K_\eps$ such that $\P (Z(t) \not\in K_\eps ) \leq \eps$ for all $t \geq 1$. 

Concerning $\bY$, we note that $\P$-almost surely
\begin{eqnarray*}
  Y(t) & = & \int_0^t S(t-s)F(Y(s) + Z(s))\, ds = \int_0^t S(r) F(X(t-r))\, dr\\
& = & \int_0^\eps S(r)F(X(t-r))\, dr + S(\eps )\int_\eps^t S(r-\eps)F(X(t-r))\, dr
\end{eqnarray*}
for all $t \geq \eps >0$.

Since $F$ is bounded and $\bS$ is uniformly exponentially stable, it is easy to see that 
there is a bounded set $B_\eps$ 
such that
\[ \int_\eps^t S(r-\eps)F(X(t-s))\, dr \in B_\eps \quad \forall\, t \geq \eps\]
Since $S(\eps)$ is compact, the set $L_\eps := S(\eps )B_\eps$ is relatively compact. We next note that 
for $\eps \leq 1$ we have 
\[ \Big\| \int_0^\eps S(r)F(X(t-r))\, dr \Big\| \leq \eps \sup_{0\leq r \leq 1} \|S(r)\|\cdot \|F\|_\infty
 =: C_1\eps
\]
Now put $S_\eps := \{ x \,:\, d(x, L_\eps ) \leq C_1\eps\}$ and $L := \bigcap_{n\geq 1}S_{n^{-1}}$. Then, almost surely,
$Y(t) \in L$ for all $t\geq 1$. Note that $L$ is closed and totally bounded. Consequently, $L$ is compact. 

It now follows that 
\[ \P (X(t) \not\in K_\eps + L ) \leq \P(Z(t) \not \in K_\eps) + \P (Y(t) \not \in L) \leq \eps + 0\]
for all $t \geq 1$. Hence the laws $\{ \mu_{X(t)}\, : \, t \geq 1\}$ are tight. Thus, by the 
Krylov-Bogoliubov theorem, there exists an invariant measure $\mu$ for $\scrP$. Since $\scrP$ is strongly Feller 
and irreducible, $\mu$ is the unique invariant measure for $\scrP$.
\end{proof}

\section{Examples}\label{sect.examples}
\subsection{The finite dimensional case}
Let us first consider the case where both $E=\tilde{E}$ and $H$ are infinite dimensional, say $E= \CR^m, H = \CR^n$.
In this case, $A$ and $G$ are matrices, $A \in \CR^{m\times m}, G \in \CR^{m\times n}$. Recalling that
 $S(t) \in \cL (E, H_t)$ for all $t>0$ implies that $H_t$ is dense in $E$, we see that $S(t) \in
\cL (E, H_t)$  if and only if $H_t = \CR^m$, i.e.\ iff $Q_t$ is invertible. It is well known that 
this is the case if and only if the \emph{Kalman rank condition} is satisfied, i.e.\ 
the $m \times nm$ matrix
\[ \big[ G, AG, \ldots , A^{m-1}G \big]\]
has full rank $m$. 

How about our second condition, that $\|S(t)\|_{\cL (E, H_t)}$ is locally integrable? 
As it turns out, in finite dimensions this condition is satisfied, if and only if the range of $G$ is 
$\CR^m$. To see this, let us first note that $\|x\|_{H_{Q_t}} = \|Q_t^{-\half}x\|$. Define 
\[ k := \min\Big\{ j \in \{0, \ldots, m-1\} \, : \, \mathrm{rank}\big[ G, \ldots, A^jG\big] = m \Big\}\, .\]
It was proved in \cite[Lemma 3.1]{lun97} that $\|Q_t^{-\half}\| = O (t^{-k-\half})$ as $t \to 0$. Hence 
our condition is satisfied if and only if $k= 0$, i.e.\ we have full noise.

\subsection{Stochastic partial differential equations}
Let $\OO$ be a bounded, open domain in $\CR^d$ with Lipschitz boundary. For $1<p<\infty$, we let $A_p$ be the 
$L^p(\OO)$-realization of the uniformly elliptic differential operator 
\[ \mathcal{A} := -\sum_{i,j=1}^d a_{ij}\frac{\partial^2}{\partial x_i\partial x_j} - 
\sum_{j=1}^d b_j\frac{\partial}{\partial x_j}\]
with domain $D(A_p) = W^{2,p}(\OO)\cap W^{1,p}_0(\OO)$, i.e.\ we require Dirichlet Boundary conditions. 
We assume that the coefficients $a_{ij} = a_{ji}$ belong 
to $C^\theta (\overline{\OO})$ for some $\theta >0$ and that the functions $b_j$ are bounded and measurable.

It is well-known that under these assumptions $-A_p$ generates a uniformly exponentially stable, analytic, 
strongly continuous semigroup $\bS_p$ on $L^p(\OO)$. Moreover, these semigroups are consistent, i.e.\ 
$S_p(t)f = S_q(t)f$ for all $t\geq 0$ whenever $f \in L^p(\OO)\cap L^q(\OO)$. 

Since $\bS_p$ is uniformly exponentially stable and analytic, the fractional powers $A_p^{-\alpha}$, their 
inverses $A_p^\alpha$ and the fractional domain spaces $D(A_p^\alpha)$ are well-defined for 
$\alpha \in (0,1)$.
We put $A_p^0 := I_{L^p(\OO)}$. 

Embedding fractional domain spaces into complex interpolation spaces, 
we see that $D(A_p^\alpha )$ continuously embeds into $W^{2\alpha -\eps, p}(\OO)$ for all $\eps >0$.\medskip 

We put $\tilde{E}=E := L^p(\OO)$ and $H= L^2(\OO)$. Thus $W_H$ is a cylindrical Wiener process on $L^2(\OO)$.
In order to inject the noise into $E$, we use a fractional power $A_2^{-\alpha}$. If $p\leq 2$, then $L^2(\OO) 
\subset L^p(\OO)$, so that we may put $G = A_2^{-\alpha}$ for any $\alpha \geq 0$. If $p >0$, we need to choose 
$\alpha$ large enough so that $D(A_2^\alpha )$ embeds into $L^p(\OO)$. By Sobolev embedding, this is the case 
whenever $\alpha > (4p)^{-1}(dp - 2d)$.

We now consider equation $[-A_p, 0, A_2^{-\alpha}]$. Then the smaller $\alpha$ becomes, the more noise is injected 
into the equation, the border case $\alpha =0$ corresponds to the stochastic equation with space-time white noise.
Thus, as $\alpha$ gets smaller it gets ``more unlikely'' that the equation is well-posed. On the other hand, 
as $\alpha$ gets smaller, it is gets ``more likely'' that the strong Feller property is satisfied.

The values for $\alpha$ for which the equation is well posed and the strong Feller property holds depends on both 
$p$ and the dimension $d$. The following Proposition collects the results.

\begin{prop}
Let $p \in (1,\infty )$.
Under the assumptions and with the notation above, assume that
$\max \{ 0, \frac{d}{4} - \half +\eps, \frac{dp-2d}{4p} +\eps\} \leq \alpha < 1$ for some $\eps >0$. 
\begin{enumerate}
 \item Equation $[-A_p, 0, A_2^{-\alpha}]$ is well-posed. 
 \item If in addition $\alpha + \max\{ 0, \frac{2d-dp}{4p} \} < \half$, then we have $S_p(t) \in \cL (E, H_{Q_t})$
with $\|S_p(t)\|_{\cL (E, H_{Q_t})} \leq \varphi (t)$ for all $t>0$ for some $\varphi \in L^1_\mathrm{loc}([0, \infty )$.
\end{enumerate}
\end{prop}

\begin{proof}
(1) The proof relies on results about stochastic integration in Banach spaces and $\gamma$-radonifying 
operators in \cite{vNW05}, to which we refer also for the definition of $\gamma$-radonifying operators. 
We give a sketch of the proof.

By \cite[Theorem 7.1]{vNW05}, we need to prove that $t \mapsto S_p(t)A_2^{-\alpha}$ represents an element of 
$\gamma (L^2(0,T;H), E)$ for some $T>0$. Since $\alpha > \frac{d}{4} -\half$, we may pick $\beta \in [0, \half)$ such that
$\alpha + \beta > \frac{d}{4}$. In this case, the  embedding $j : H^{2(\alpha + \beta )}(\OO) \to L^p(\OO)$
is $\gamma$-radonifying, see \cite[Corollary 2.2]{vNVW08}. By consistency, we may write 
$S_p(t)A_2^{-\alpha} =  S_p(t) A_p^\beta j A_2^{-(\alpha + \beta)}$, where we consider $A_p^\beta$ as an operator 
from $E$ to the extrapolation space $E_{-\beta}$.

By \cite[Lemma 4.1]{vNVW08}, the set $\{ t^{\beta + \delta}S_p(t) \, : \, t \in (0,T] \} 
\subset \cL (E_{-\beta}, E)$ is $\gamma$-bounded with $\gamma$-bound of order $O(T^\delta)$ as $T\to 0$. 
Since $t \mapsto t^{-(\beta + \delta)}jA_2^{-(\alpha + \beta)}$ represents an element of $\gamma (L^2(0,T;H), E)$, 
a multiplier result in \cite{kwpre}, see also \cite[Lemma 2.9]{vNVW08}, proves that $S_p(\cdot)A_2^{-\alpha}$ 
represents an element 
of $\gamma (L^2(0,T; H), E)$. Moreover, its norm is of order $O(T^\delta)$ as $\eps \to 0$. Hence, by \cite{vN10}, 
the stochastic convolution $\int_0^\cdot S(\cdot - s)A_2^{-\alpha}dW_H(s)$ has a continuous modification.
It follows that $[-A_p, 0, A_2^{-\alpha}]$ has a unique mild martingale solution for any initial value.\medskip 

(2)  For $f \in L^p(\OO)$, we have 
\begin{eqnarray*}
 S_p(t)f & = & \frac{1}{t}\int_0^t S_p(t)f\, ds = \int_0^t S_p(t-s)\frac{1}{t}S_p(s)f\, ds\\
& = &  \int_0^t S_p(t-s)A_2^{-\alpha}\frac{1}{t}A_p^{\alpha}S_p(s)f\, ds 
\end{eqnarray*}
In the language of control theory, $u_f:= t^{-1}A_p^\alpha S_p(\cdot )f$ is a control for reaching $S_p(t)f$ 
at time $t$. Thus, by the results of \cite{vN05}, if $u_f$ is an element of $L^2(0,t; L^2(\OO))$, then 
$S(t)f \in H_{Q_t}$ and $\|S(t)f\|_{H_{Q_t}} \leq \|u_f\|_{L^2(0,t;L^2(\OO))}$.

If $p \geq 2$, then, since $\alpha< \half$, we have $u_f \in L^2(0,t; L^p(\OO)) \subset L^2(0,t; L^2(\OO))$. 
Otherwise, by the additional assumption, we may pick
$\beta \geq 0$ such that $\alpha + \beta < \half$ and $D(A_p^\beta)$ continuously embeds into 
$L^2(\OO)$. We obtain
\begin{eqnarray*}
 \|u_f\|_{L^2(0,t;H)}^2 & = & t^{-2}\int_0^t \|A_p^\alpha S(s)f\|_{L^2(\OO)}^2\, ds 
 \lesssim  t^{-2}\int_0^t \|A_p^{\alpha }S(s)f\|_{D(A^{\beta})}^2\, ds\\
& \leq & t^{-2}\int_0^t \|A_p^{\alpha + \beta} S_p(s)f\|_{L^p(\OO)}\, ds 
\leq t^{-2}\int_0^t s^{-2(\alpha+\beta)}\, ds \|f\|_{L^p(\OO)}\\
& \lesssim & t^{-1-2(\alpha + \beta)} \|f\|_{L^p(\OO)}\, .
\end{eqnarray*}
Hence, $\|S(t)\|_{\cL (E, H_{Q_t})} \lesssim 
t^{-\half - (\alpha + \beta )}$ and, since $\alpha + \beta < \half$, the latter belongs to $L^1_\mathrm{loc}([0,\infty ))$.
\end{proof}

\subsection{1-dimensional equation with space-time white noise}
We now look in more detail to the stochastic partial differential equation from the previous section 
in one space dimension with space-time white noise, which is of great importance in applications. 
For convenience, we set $\OO = (0,1)$.

Thus, our stochastic partial differential equation takes the form
\[
 \left\{
\begin{array}{lll}
\frac{\partial u}{\partial t}(t,x)& = &
a(x) \frac{\partial^2 u}{\partial x^2}(t,x) + b(x)\frac{\partial u}{\partial x}(t,x)
+ c(x) + \frac{\partial w}{\partial x}(t,x)\\[0.4em]
& & \quad \quad \mbox{for}\quad t>0, x \in (0,1)\\
u(t,x) & = & 0 \quad\mbox{for}\quad t>0, x \in \{0,1\}\\
u(0,x) & = & u_0(x) \quad \mbox{for}\quad x \in (0,1) ,
\end{array}
\right. 
\]
where $w$ is a space-time white noise.

The results of the previous subsection for $\alpha =0$ and $d=1$ show that for $p \in (1,2)$ this equation is well posed 
on $L^p(0,1)$ and the corresponding transition semigroup satisfies the Hypothesis of Theorem \ref{t.main}. For $p>2$, 
we may put $\tilde{E}= L^2(0,1)$. Note that $\|S(t)\|_{\cL(L^2(0,1), H^{2\theta}(0,1))} \lesssim t^{-\theta} \in L^2_{\mathrm{loc}}
(0,\infty)$ for $\theta < \half$. Hence, given $p>2$, picking $\theta$ close enough to $\frac{1}{4}$ it follows from 
Sobolev embedding that $\|S(t)\|_{\cL (\tilde{E}, E)} \in L^2_{\mathrm{loc}}(0,\infty)$ hence \eqref{eq.hyp2} is satisfied in this case.
Since $L^p(0,1)$ is continuously embedded into $L^2(0,1)$, estimate \eqref{eq.hyp1} follows from the corresponding estimate
on $L^2(0,1)$. The well-posedness of the equation on $L^p(0,1)$ follows with similar arguments as above, we leave 
the details to the reader.

It is also possible to consider the equation on the state space $E = C_0(0,1)$ of continuous functions vanishing at the 
boundary. Indeed, first considering the equation on $L^p(0,1)$, the factorization method, cf.\ \cite[Corollary 4.4]{vNVW08}
for the Banach space situation, yields that for $u_0 \in D(A_p^\theta)$ the solution of the equation 
has a modification with continuous paths in $D(A_p^\eta)$ for $\eta < \half$. On the other hand, since $D(A_p^\eta)
\inject W^{2\eta, p}_0(0,1)$, the Sobolev embedding yields 
$D(A_p^\eta) \subset C_0(0,1)$ whenever $2\eta > \frac{1}{p}$. 
Consequently, for $\eta \in (\frac{1}{2p},\half)$ 
the equation is well-posed on $C_0(0,1)$ for initial values $u_0 \in D(A_p^\theta)$. With standard arguments, this can be 
extended to initial values in $C_0(0,1)$, hence the equation is well-posed on $C_0(0,1)$. Note that 
\[
 \|S(t)\|_{\cL (\tilde{E}, E)} \leq \|S(t/2)\|_{\cL (L^2(0,1), L^p(0,1))}\|S(t/2)\|_{\cL (L^2(0,1), D(A_p^\theta))}
\lesssim t^{-(\theta+\eta)}\, .
\]
Thus, by picking $p$ large enough, we can pick $\theta$ and $\eta$ such that $\theta+\eta < \half$, i.e.\ such that 
\eqref{eq.hyp2} is satisfied.

Moreover, since \eqref{eq.hyp1} is satisfied on $L^p(0,1)$, it follows as above that \eqref{eq.hyp1} is satisfied 
on $C_0(0,1) \subset L^p(0,1)$. Consequently, the assumptions of Theorem \ref{t.main} are satisfied and our results 
yield 

\begin{thm}
Let $a \in C^\delta([0,1])$ for some $\delta>0$, $b,c \in L^\infty (0,1)$ and $F : C_0(0,1) \to C_0(0,1)$ be bounded 
and measurable. Then the equation 
\[
 \left\{
\begin{array}{lll}
\frac{\partial u}{\partial t}(t,x)& = &
a(x) \frac{\partial^2 u}{\partial x^2}(t,x) + b(x)\frac{\partial u}{\partial x}(t,x)
+ c(x) + \big[F(u(t))\big](x) + \frac{\partial w}{\partial x}(t,x)\\[0.4em]
& & \quad \quad \mbox{for}\quad t>0, x \in (0,1)\\
u(t,x) & = & 0 \quad\mbox{for}\quad t>0, x \in \{0,1\}\\
u(0,x) & = & u_0(x) \quad \mbox{for}\quad x \in (0,1) ,
\end{array}
\right. 
\] 
Is well-posed on $E=C_0(0,1)$. Moreover, the corresponding transition semigroup is strongly Feller and irreducible 
and admits a unique invariant measure.
\end{thm}

\section*{Acknowledgment}
I would like to thank Jan van Neerven for reading a large part of this article and for several 
useful suggestions.

\end{document}